\documentclass[12pt]{elsarticle}
\usepackage{amsthm,amsmath,amssymb,
MnSymbol,easymat,mathrsfs}
\usepackage[active]{srcltx}
\DeclareMathOperator{\diag}{diag}
\DeclareMathOperator{\im}{Im}
\DeclareMathOperator{\re}{Re}

\renewcommand{\le}{\leqslant}
\renewcommand{\ge}{\geqslant}

\begin{document}
\title{Miniversal
deformations of
matrices under *congruence and
reducing transformations}

\author[dmi]{Andrii Dmytryshyn}
\ead{andrii@cs.umu.se}
\address[dmi]{Department of Computing Science
and HPC2N, Ume{\aa} University, Sweden}

\author[fut]{Vyacheslav Futorny}
\ead{futorny@ime.usp.br}
\address[fut]{Department of Mathematics, University
of S\~ao Paulo, Brazil}

\author[ser]{Vladimir V.
Sergeichuk\corref{cor}}
\ead{sergeich@imath.kiev.ua}
\address[ser]{Institute of Mathematics,
Tereshchenkivska 3,
Kiev, Ukraine}
\cortext[cor]{Corresponding author}

\newtheorem{theorem}{Theorem}[section]
\newtheorem{lemma}{Lemma}[section]

\theoremstyle{definition}
\newtheorem{definition}{Definition}[section]

\theoremstyle{remark}
\newtheorem{remark}{Remark}[section]
\newtheorem{corollary}{Corollary}[section]
\newtheorem{example}{Example}[section]

%\begin{document}

\begin{abstract}
V.I. Arnold [Russian Math. Surveys 26
(2) (1971) 29--43] constructed a
miniversal deformation of a square
complex matrix under similarity; that
is, a simple normal form to which not
only a given square matrix $A$ but all
matrices $B$ close to it can be reduced
by similarity transformations that
smoothly depend on the entries of $B$.
We give miniversal deformations of
matrices of sesquilinear forms; that
is, of square complex matrices under
*congruence, and construct an analytic
reducing transformation to a miniversal
deformation. Analogous results for
matrices under congruence were obtained
by the authors in [Linear Algebra Appl.
436 (2012) 2670--2700].
\end{abstract}

\begin{keyword}
Sesquilinear forms\sep *Congruence
transformations\sep Miniversal
deformations

\MSC 15A21\sep 15A63\sep 47A07
\end{keyword}

 \maketitle

\section{Introduction}

The reduction of a matrix to its Jordan
form is an unstable operation: both the
Jordan form and the reduction
transformation depend discontinuously
on the entries of the original matrix.
Therefore, if the entries of a matrix
are known only approximately, then it
is unwise to reduce it to Jordan form.
Furthermore, when investigating a
family of matrices smoothly depending
on parameters, then although each
individual matrix can be reduced to a
Jordan form, it is unwise to do so
since in such an operation the
smoothness relative to the parameters
is lost.

For these reasons, Arnold \cite[Theorem
4.4]{arn} (see also \cite{arn2,arn3})
constructed miniversal deformations of
complex matrices under similarity; that
is, a simple normal form to which not
only a given square matrix $A$ but all
matrices $B$ close to it can be reduced
by similarity transformations that
smoothly depend on the entries of $B$.
Miniversal deformations were also
constructed
\begin{itemize}
  \item for real matrices with
      respect to similarity by
      Galin \cite{gal} (see also
      \cite{arn2,arn3}) and
      simplified by Garcia-Planas
      and Sergeichuk
      \cite{gar_ser};

  \item for complex matrix pencils
      by Edelman, Elmroth, and
      K\r{a}gstr\"{o}m \cite{kag};
      for complex and real matrix
      pencils and contragredient
      matrix pencils (i.e., pairs
      of linear maps
      $U\rightleftarrows V$) by
      Garcia-Planas, Klimenko, and
      Sergeichuk
      \cite{gar_ser,k-s_triang};

   \item for matrices up to
       congruence by the authors
       \cite{f_s}; for pairs of
       skew-symmetric or symmetric
       matrices up to congruence by
       Dmytryshyn \cite{dm1,dm2};

\item for matrices up to unitary
    similarity by Benedetti and
Cragnolini \cite{ben}; for matrices
of selfadjoint operators on a
complex
       or real vector space with
       scalar product given by a
       skew-symmetric, or
       symmetric, or Hermitian
       nonsingular form in
       \cite{djo,gal2,pat1,pat3}.
       Deformations of Hermitian
       matrices were
studied by von Neumann and Wigner
\cite{von}.

\end{itemize}

All matrices that we consider are
complex matrices.

The main results of this paper are the
following:
\begin{itemize}
  \item we construct a miniversal
      deformation of a square
      complex matrix $A$ with
      respect to \emph{*congruence
      transformations} $S^*AS$ ($S$
      is nonsingular); i.e., we
      give a simple normal form
      $B_{\text{mindef}}$ to which
      all matrices $B$ close to $A$
      can be reduced by a
      *congruence transformation
that is an analytic function of the
entries of both $B$ and its complex
conjugate $\bar B$;

  \item we construct this analytic
      *congruence transformation.
\end{itemize}

Applications of Arnold's miniversal
deformations of matrices under
similarity are based, in particular, on
the fact that a matrix and its Arnold
normal form have the same eigenvalues
\cite{mai2000,mai2001}. In a similar
way, possible applications of the
normal form $B_{\text{mindef}}$ can be
based on the fact that $B$ and
$B_{\text{mindef}}$ have the same
invariants with respect to *congruence,
while $B_{\text{mindef}}$ has a very
simple structure.  A preliminary
version of this article was used in
\cite{f_*diagr} for constructing the
Hasse diagram of the closure ordering
on the set of *congruence classes of
$2\times 2$ matrices.

Without loss of generality, we assume
that $A$ is a canonical matrix for
*congruence; we use the *congruence
canonical matrices given by Horn and
Sergeichuk \cite{hor-ser}.

For the reader convenience, we first
give the normal form
$B_{\text{mindef}}$ in terms of
analytic matrix functions in Theorem
\ref{teo2} of Section \ref{s2}. It is
formulated in terms of miniversal
deformations in Section \ref{sect2}.
The parameters of miniversal
deformation of a square matrix $A$ are
independent and real; their number is
equal to the codimension over $\mathbb
R$ of the *congruence class of $A$;
this codimension was calculated by De
Ter\'an and Dopico \cite{t_d-1} (as the
codimension of the *congruence orbit of
$A$). The codimension of the congruence
class of a square matrix was calculated
by De Ter\'an and Dopico \cite{t_d}
and, independently, by the authors
\cite{f_s}. The codimensions of the
congruence classes of a pair of
skew-symmetric matrices and a pair of
symmetric matrices were calculated by
Dmytryshyn, K\r{a}gstr\"{o}m, and
Sergeichuk \cite{par-cos,par-symm}.

We prove Theorem \ref{teo2} in Section
\ref{kku} (note that the proofs in
Sections \ref{sub1}, \ref{sub4}, and
\ref{sub5} are very close to the proofs
in \cite[Sections 5.3, 6.3, 7.2]{f_s}).
The proof is based on Theorem
\ref{ttft}, which gives a method for
constructing both a miniversal
deformation and an analytic *congruence
transformation that reduces all
matrices in a neighborhood of a given
square matrix $A$ to the miniversal
deformation of $A$. A knowledge of this
transformation is important for
applications of miniversal
deformations.

Methods of constructing transformations
to miniversal deformations were
developed for matrices under similarity
and for matrix pencils by Garcia-Planas
and Mailybaev
\cite{gar_mai,mai2000,mai2001} (see
also Schmidt \cite{sch,sch1} and
Stolovitch \cite{sto}) and for matrices
under congruence by the authors
\cite{f_s}; these transformations are
analytic functions of the entries of
matrices or matrix pencils in a
neighborhood of a given matrix or
pencil. Unlike them, the *congruence
transformation $B\mapsto
B_{\text{mindef}}$ that we construct in
Section \ref{sect4} is an analytic
function of the entries of both $B$ and
$\bar B$. Klimenko \cite{klim} proved
that the words ``and its complex
conjugate $\bar B$'' cannot be deleted:
the reducing transformation that is an
analytic function only of the entries
of $B$ does not exist even if $A$ is
the $1\times 1$ identity matrix $I_1$.
Recall that the complex conjugate
function $z\mapsto \bar z$ on $\mathbb
C$ is not analytic at $0$.

\section{The main
theorem in terms of analytic
functions}\label{s2}

Define the $n\times n$ matrices:
\begin{equation*}\label{1aa}
J_n(\lambda):=\begin{bmatrix}
\lambda&1&&0\\&\lambda&\ddots&\\&&\ddots&1
\\ 0&&&\lambda
\end{bmatrix},\qquad
\Delta_{n}=%
\begin{bmatrix}
0 &  &  & 1\\
&  &
\udots & i\\
& 1 &
\udots & \\
1 & i &  & 0
\end{bmatrix}.
\end{equation*}

We use the following canonical form of
complex matrices for *congruence.

\begin{theorem}[{\cite[Theorem 4.5.21]{hor_John}}]
\label{t2a}  Each square complex matrix
is *congruent to a direct sum, uniquely
determined up to permutation of
summands, of matrices of the three
types:
\begin{equation}\label{table2}
H_{2m}(\lambda ):=
\begin{bmatrix}0&I_m\\
J_m(\lambda) &0
\end{bmatrix}\
(\lambda \in\mathbb C,\ |\lambda |>1),
\ \ \mu \Delta_{n}\
(\mu \in\mathbb C,\ |\mu |=1) ,\ \
J_k(0).
\end{equation}
\end{theorem}
This canonical form was obtained in
\cite{hor-ser} basing on \cite[Theorem
3]{ser_izv} and was generalized to
other fields in
\cite{hor-ser_anyfield}. A direct proof
that this form is canonical is given in
\cite{hor-ser_regul, hor-ser_can}.

A matrix in which each entry is $0$,
$*$, $\circ$, or $\bullet$ is called a
\emph{$(0,\!*,\!\circ,\!\bullet)$
matrix}. We use the following
$(0,\!*,\!\circ,\!\bullet)$ matrices:

$\bullet$ The $m\times n$ matrices
\[0^{\swarrow}:=\left[\begin{MAT}(e){cccc}
*&&&\\
\vdots&&0&\\
*&&&\\
\end{MAT}\right]\text{ if $m\le n$ or }
\left[\begin{MAT}(e){ccc}
&&\\
&0&\\
&&\\
*&\cdots&*\\
\end{MAT}\right] \text{ if $m\ge n$},
\]
                  %%%%%%%%%
\[0^{\swvdash}:=\left[\begin{MAT}(b){cccccccc}
\vdots&&&&&&&\\
0&&&&&&&\\
*&&&&0&&&\\
0&&&&&&&\\
*&&&&&&&
\\
\end{MAT}\right]\text{ if $m\le n$ or }
\left[\begin{MAT}(e){ccccc}
&&&&\\
&&&&\\
&&0&&\\
&&&&\\
&&&&\\
&&&&\\
*&0&*&0&\cdots\\
\end{MAT}\right]\text{ if $m\ge n$}
\]
(choosing among the left and right
matrices in these equalities, we take a
matrix with the minimum number of
stars; we can take any of them if
$m=n$).

$\bullet$ The matrices
\[0^{\nwarrow},\quad 0^{\nearrow}, \quad\text{and}\quad 0^{\searrow}
\] that are obtained by
rotating $0^{\swarrow}$ by $90$, $180$,
and $270$ degrees clockwise.

$\bullet$ The $n\times n$ matrices
\begin{align*}\label{ddw}
0^{\sespoon}:=
  \begin{cases}
    \diag(*,\dots,*,0,\dots,0)& \text{if $n=2k$}, \\
    \diag(*,\dots,*,\circ,0,\dots,0)& \text{if $n=2k+1$},
  \end{cases}\\
0^{\sefilledspoon}:=
  \begin{cases}
    \diag(*,\dots,*,0,\dots,0)& \text{if $n=2k$}, \\
    \diag(*,\dots,*,\bullet,0,\dots,0)&
\text{if $n=2k+1$},
  \end{cases}
\end{align*}
in which $k$ is the number of $*$'s.

$\bullet$ The $m\times n$ matrices
\begin{equation*}\label{bjhf}
0^{\updownarrow}:=\left[\begin{MAT}(b){ccc}
*&\cdots&*\\
&&\\
&0&\\
&&\\
\end{MAT}\right]\quad \text{or}\quad
\left[\begin{MAT}(b){ccc}
&&\\
&0&\\
&&\\
*&\cdots&*\\
\end{MAT}\right]\end{equation*}
($0^{\updownarrow}$ can be taken in any
of these forms), and
\begin{equation}\label{hui}
{\cal P}_{mn}:=\begin{bmatrix}
\begin{matrix}
0&\dots& 0
  \\
\vdots&\ddots& \vdots
\end{matrix}&0\\
\begin{matrix}
0& \dots&
0\end{matrix}&
\begin{matrix}
0\ *\ \dots\ *
\end{matrix}
\end{bmatrix}\quad
\text{in which } m\le n
\end{equation}
(${\cal P}_{mn}$ has $n-m-1$ stars if
$m<n$).

\begin{definition}\label{kkj}
Let
$A_{\text{can}}=A_1\oplus\dots\oplus
A_t$ be a *congruence canonical matrix,
in which all summands $A_i$ are of the
form \eqref{table2} and are arranged
thus:
\begin{equation}\label{juw}
A_{\text{can}}=
\bigoplus_i
H_{2p_i}(\lambda_i)
 \oplus
\bigoplus_j\mu_j\Delta_{q_j}
 \oplus
\bigoplus_l
J_{r_l}(0),\qquad
r_1\ge r_2\ge\dots
\end{equation}
Let us construct a
$(0,\!*,\!\circ,\!\bullet)$ matrix
\begin{equation*}\label{grsd}
{\cal D}\langle A_{\text{can}}\rangle
=\begin{bmatrix}
{\cal
D}_{11}&\dots&{\cal
D}_{1t}
 \\
\vdots&\ddots&\vdots\\
{\cal
D}_{t1}&\dots&{\cal
D}_{tt}
\end{bmatrix}
\end{equation*}
partitioned conformally to the
partition of $A_{\text{can}}$. Its
blocks ${\cal D}_{ij}$ are defined by
the following equalities, in which
\begin{equation*}\label{lhs}
{\cal D}\langle A_i\rangle:={\cal
D}_{ii},\qquad {\cal
D}\langle A_i,A_j\rangle  :=({\cal
D}_{ji},{\cal
D}_{ij})\  \text{with
}i<j,
\end{equation*}
$|\lambda|>1$, $|\lambda'|>1$, and
$|\mu |=|\mu'|=1$:
\begin{itemize}
  \item[{\rm(i)}] The diagonal
      blocks of ${\cal D}\langle
      A_{\text{can}}\rangle $ are
      defined by
\begin{align}\label{KEV}
{\cal
D}\langle H_{2m}(\lambda)\rangle &=
    \begin{bmatrix}
0&0
 \\ 0^{\swarrow}&0
\end{bmatrix},
   %%%%
     \\ \nonumber
   %%%%
{\cal D}\langle \mu \Delta_n\rangle &=
  \begin{cases}
    0^{\sespoon}& \text{if }
\mu\notin\mathbb R, \\
    0^{\sefilledspoon} &
    \text{if }\mu\in \mathbb R,
  \end{cases} \begin{matrix}
\text{($0^{\sefilledspoon}$ can be used
\qquad\qquad} \\
\text{\qquad instead of
$0^{\sespoon}$ if $\mu\notin \mathbb R\cup
i\mathbb R$)}
\end{matrix}
   %%%%
     \\
  \nonumber
   %%%%
{\cal
D}\langle J_n(0)\rangle &=0^{\swvdash}
\end{align}

  \item[{\rm(ii)}] Each
      off-diagonal block of ${\cal
      D}\langle
      A_{\text{can}}\rangle $ whose
      horizontal and vertical
      strips contain summands of
      $A_{\text{\rm can}}$ of
      \emph{the same type} is
      defined by
\begin{align}
\label{lsiu1}
{\cal D}
\langle H_{2m}(\lambda),\,
H_{2n}(\lambda')\rangle
       &=
  \begin{cases}
(0,\:0) &\text{if
$\lambda\ne\lambda'$},
          \\[1mm]
    \left(\begin{bmatrix}
0&0^{\nearrow}
 \\ 0^{\swarrow}&0
\end{bmatrix},\:0
\right)
 &\text{if $\lambda=\lambda'$},
  \end{cases}
%%%%
     \\ \nonumber
   %%%%
{\cal D}
\langle \mu \Delta_m, \mu' \Delta_n\rangle &=
  \begin{cases}
    (0,\:0) & \text{if }\mu\ne\pm\mu', \\
    (0^{\nwarrow},\:0) & \text{if }\mu=
\pm\mu',
  \end{cases}
%%%%
     \\ \nonumber
   %%%%
{\cal D}
\langle J_m(0),J_n(0)\rangle &=
  \begin{cases}
(0^{\swvdash},\:
0^{\swvdash})
 & \text{if $m\ge n$ and $n$ is even},
    \\
(0^{\swvdash}+{\cal
P}_{nm},\:0^{\swvdash})
 & \text{if $m\ge n$ and $n$ is odd}.
  \end{cases}
\end{align}

  \item[{\rm(iii)}] Each
      off-diagonal block of ${\cal
      D}\langle
      A_{\text{can}}\rangle $ whose
      horizontal and vertical
      strips contain summands of
      $A_{\text{\rm can}}$ of
      \emph{distinct types} is
      defined by
\begin{align*}
{\cal D}
\langle H_{2m}(\lambda),\mu \Delta_n\rangle &= (0,\: 0),
                     \\
{\cal D}
\langle H_{2m}(\lambda),J_n(0)\rangle =
{\cal D}
\langle \mu \Delta_m,J_n(0)\rangle &=
  \begin{cases}
(0,\: 0)
 & \text{if $n$ is even}, \\
(0^{\updownarrow},\:0)
 & \text{if $n$ is odd}.
  \end{cases}
\end{align*}
\end{itemize}
\end{definition}

For each $(0,\!*,\!\circ,\!\bullet)$
matrix $\cal D$, we denote by
\begin{equation}\label{kud}
{\cal
D}({\mathbb C})
\end{equation}
the real vector space of all matrices
obtained from $\cal D$ by replacing its
entries $*$, $\circ$, and $\bullet$ in
${\cal D}$ by complex, real, and pure
imaginary numbers, respectively.
Clearly, $\dim_{\mathbb R}{\cal
D}({\mathbb C})$ is equal to the number
of circles and bullets plus twice the
number of stars in ${\cal D}$.

Locally speaking, a function of several
variables is \emph{analytic} if it can
be given as a power series in those
variables. A matrix is said to be an
\emph{analytic function} of several
complex (or real) parameters if its
entries are analytic functions of these
parameters.

The following theorem is our main
result; it will be reformulated in
terms of miniversal deformations in
Theorem \ref{teojy}.

\begin{theorem}\label{teo2}
Let $A_{\text{\rm can}}$ be a canonical
matrix for *congruence, and let ${\cal
D}:={\cal D}\langle A_{\text{\rm
can}}\rangle $ be its
$(0,\!*,\!\circ,\!\bullet)$ matrix from
Definition \ref{kkj}. Then there exists
a neighborhood $U$ of\/ $0_n$ in
$\mathbb C^{n\times n}$ such that all
matrices $A_{\text{\rm can}}+X$ with
$X\in U$ are simultaneously reduced by
some transformation
\begin{equation}\label{tef}
{\cal
S}(X)^* (A_{\text{\rm can}}+X) {\cal
S}(X),\quad\begin{matrix}
\text{${\cal S}:U\to \mathbb
C^{n\times n}$
is an analytic function}\\
\text{of the entries of $X\in U$ and its
complex}\\
\text{conjugate $\bar X$, ${\cal
S}(X)$ is nonsingular for}\\
 \text{all $X\in U$, and
${\cal
S}(0_n)=I_n$},
\end{matrix}
\end{equation}
to matrices from $A_{\text{\rm can}}
+{\cal D}(\mathbb C)$; their entries
are analytic functions of the entries
of $X$ and $\bar X$ on $U$. The number
$\dim_{\mathbb R}{\cal D}({\mathbb C})$
is the smallest that can be made by
transformations \eqref{tef}; it is
equal to the codimension over $\mathbb
R$ of the *congruence class of
$A_{\text{\rm can}}$.
\end{theorem}

\begin{remark}\label{ksy}
The transforming matrix $S(X)$ in
\eqref{tef} does not always can be
taken as an analytic function of the
entries of $X$: if
$A_{\text{can}}=I_1$, then
$A_{\text{can}} +{\cal D}\langle
A_{\text{can}}\rangle =I_1+[\bullet]$
and by Theorem \ref{teo2} there exist a
neighborhood $U$ of $0$ and continuous
mappings $s:U\to \mathbb C$ and
$\varphi: U\to i\mathbb R$ such that
$s(0)=1$ and
\begin{equation*}\label{kir}
[s(x)]^*[1+x][s(x)]=[1+\varphi(x)]
\end{equation*}
for all $x\in U$. Klimenko
{\rm\cite{klim}} proved that $s$ cannot
be an analytic function. In contrast to
this, the transforming matrices for
matrices under similarity
{\rm\cite{arn}}, matrix pencils under
strict equivalence {\rm\cite{kag}}, and
matrices under congruence
{\rm\cite{f_s}} always can be taken
analytic.
\end{remark}

\begin{remark}\label{Hks}
The transforming matrix $S(X)$ in
\eqref{tef} can be considered as an
analytic function of the entries of the
real and imaginary parts of $X$ since
$X=\re X+i\im X$ and $\bar X=\re X-i\im
X$.
\end{remark}

\begin{remark}\label{HRe}
Theorem \ref{teo2} is easily extended
to all pairs of Hermitian matrices that
are sufficiently close to a given pair
of Hermitian matrices (i.e., to pairs
of Hermitian forms that are close to a
given pair of Hermitian forms). All one
has to do is to express all matrices
from $A_{\text{\rm can}} +{\cal
D}(\mathbb C)$ as the sum ${H}+i{K}$ in
which ${H}$ and ${K}$ are Hermitian
matrices. The canonical Hermitian pairs
$({H}_{\text{\rm can}}, {K}_{\text{\rm
can}})$ such that ${H}_{\text{\rm
can}}+ i{K}_{\text{\rm
can}}=A_{\text{\rm can}}$ were
described in {\rm\cite[Theorem
1.2(b)]{hor-ser_can}}.
\end{remark}

For each $A\in{\mathbb C}^{n\times n}$
and a small matrix $X\in{\mathbb
C}^{n\times n}$,
\[
(I+X)^*A(I+X)
=A+\underbrace{X^* A+ AX}
_{\text{small}}
+\underbrace{X^* AX}
_{\text{very small}}
\]
and so the *congruence class of $A$ in
a small neighborhood of $A$ can be
obtained by a very small deformation of
the real affine matrix space $\{A+ X^*
A+ AX\,|\,X\in{\mathbb C}^{n\times
n}\}$. (By the local Lipschitz property
\cite{rodm}, the transforming matrix
$S$ to a matrix $S^*AS$ near $A$ can be
taken in the form $I+X$ with a small
$X$.) The real vector space
\[
T(A):=\{X^*A+AX\,|\,X\in{\mathbb
C}^{n\times n}\}
\]
is the tangent space to the *congruence
class of $A$ at the point $A$.  The
numbers $\dim_{\mathbb R} T(A)$ and
$2n^2-\dim_{\mathbb R} T(A)$ are called
the \emph{dimension} and, respectively,
\emph{codimension} over ${\mathbb R}$
of the *congruence class of $A$.

For each $A\in{\mathbb C}^{n\times n}$
there exists a
$(0,\!*,\!\circ,\!\bullet)$ matrix
${\cal D}$ such that ${\mathbb
C}^{\,n\times n}=T(A) \oplus_{\mathbb
R} {\cal D}({\mathbb C})$ because there
exists a real space $V$ generated by
matrices of the form $E_{kl}$ and
$iE_{kl}$ ($E_{kl}$ are the matrix
units) such that ${\mathbb
C}^{\,n\times n}=T(A) \oplus_{\mathbb
R} V$.

\emph{We prove Theorem \ref{teo2} as
follows:} first we show in Theorem
\ref{ttft} that each
$(0,\!*,\!\circ,\!\bullet)$ matrix
${\cal D}$ that satisfies
\begin{equation}\label{jyr}
{\mathbb C}^{\,n\times
n}=T(A_{\text{can}})
 \oplus_{\mathbb R} {\cal
D}({\mathbb C})
\end{equation}
can be taken instead of ${\cal
D}\langle A_{\text{can}}\rangle $ in
Theorem \ref{teo2}; then we verify in
Section \ref{kku} that ${\cal D}={\cal
D}\langle A_{\text{can}}\rangle $ from
Definition \ref{kkj} satisfies
\eqref{jyr}. The latter imply that
$\dim_{\mathbb R}{\cal D}({\mathbb C})$
is the \emph{codimension over $\mathbb
R$ of the *congruence class of
$A_{\text{\rm can}}$}; it was
previously calculated by De Ter\'an and
Dopico \cite{t_d-1}. Miniversal
deformations of matrix pencils and
contagredient matrix pencils and of
matrices under congruence were
constructed in \cite{f_s,gar_ser} by
analogous methods.

The method of constructing miniversal
deformations that is based on a direct
sum decomposition analogous to
\eqref{jyr} was developed by Arnold
\cite{arn} for matrices under
similarity (see also \cite{arn2} and
\cite[Section 1.6]{arn4}). It was
generalized by Tannenbaum \cite[Part V,
Theorem 1.2]{tan} to a Lie group acting
on a complex manifold (see also
\cite[Theorem 2.3]{gar_mai}, and
\cite[Theorem 2.1]{gar_ser} for
deformations of quiver
representations).  In these cases the
reducing transformations can be taken
analytic (without complex conjugation
of matrix entries as in \eqref{tef}).

\begin{example}
    \label{colh}
Let $A$ be any $2\times 2$ matrix. Then
all matrices $A+X$ that are
sufficiently close to $A$ can be
simultaneously reduced by
transformations $ {\cal S}(X)^* (A+X)
{\cal S}(X)$, in which ${\cal S}(X)$ is
an analytic function of the entries of
$X$ and $\bar X$ in a neighborhood of
$0_2$ and ${\cal S}(0_2)$ is
nonsingular, to one of the following
forms:
\begin{align*}
&\begin{bmatrix} 0&0\\
0&0
 \end{bmatrix}+
\begin{bmatrix}
 *&*\\ *&*
 \end{bmatrix},
                            &&
\begin{bmatrix} \lambda &0\\
0&0
 \end{bmatrix}+
\begin{bmatrix}
\varepsilon_{\lambda}&0\\ *&*
 \end{bmatrix}\ (|\lambda |=1) ,
                         \\&
\begin{bmatrix} \lambda &0\\
0&\pm\lambda
 \end{bmatrix}+
\begin{bmatrix}
\varepsilon_{\lambda}&0\\ *&
\delta_{\lambda}
 \end{bmatrix}\ (|\lambda |=1),
                         &&
\begin{bmatrix} \lambda &0\\
0&\mu
 \end{bmatrix}+
\begin{bmatrix}
 \varepsilon_{\lambda}&0\\
 0&\delta_{\mu}
 \end{bmatrix}
\begin{array}{l}
(\lambda \ne\pm\mu, \\
\ \ |\lambda|=|\mu |= 1),  \\
\end{array}
                       \\&
\begin{bmatrix} 0&1\\
\lambda  &0
 \end{bmatrix}+
\begin{bmatrix}
 0&0\\ *&0
 \end{bmatrix}\ (|\lambda | < 1),
                         &&
\begin{bmatrix} 0&\lambda \\
\lambda  &\lambda i
 \end{bmatrix}+
\begin{bmatrix}
 *&0\\ 0&0
 \end{bmatrix}\ (|\lambda|=1).
 \hspace{-12pt}
\end{align*}
Each of these matrices has the form
$A_{\rm can}+{D}$, in which $A_{\rm
can}$ is a direct sum of blocks of the
form \eqref{table2}; the stars in ${D}$
are complex numbers; the numbers
$\varepsilon_{\nu}$ and $\delta_{\nu}$
for each $\nu\in\mathbb C$ are either
real if $\nu
 \notin\mathbb R$ or pure imaginary
if $\nu \in\mathbb R$. (Clearly, ${D}$
tends to zero as $X$ tends to zero.)
For each $A_{\rm can}+{\cal D}$, twice
the number of its stars plus the number
of its entries of the form
$\varepsilon_{\lambda},\delta_{\lambda},
\delta_{\mu}$ is equal to the
codimension over $\mathbb R$ of the
*congruence class of $A_{\rm can}$.
\end{example}

\section{The main theorem in terms of
deformations} \label{sect2}

The purpose of this section (which is
not used in the rest of the paper) is
to extend Arnold's notion of miniversal
deformations for matrices under
similarity to matrices under
*congruence and to reformulate Theorem
\ref{teo2} in these terms. Arnold
\cite{arn,arn3} defines a deformation
of an $n\times n$ matrix $A$ as an
analytic map ${\cal A}: (\mathbb
C^r,\underline 0) \to ({\mathbb
C}^{n\times n},A)$ from a neighborhood
of $\underline 0=(0,\dots,0)$ in
$\mathbb C^r$ to ${\mathbb C}^{n\times
n}$ such that ${\cal A}(\underline
0)=A$; Remarks \ref{ksy} and \ref{Hks}
force us to define a deformation in the
case  of *congruence as an analytic map
from a neighborhood of $\underline 0$
in $\mathbb R^r$ to ${\mathbb
C}^{n\times n}$.

An \emph{$\mathbb R$-deformation} of a
matrix $A\in{\mathbb C}^{n\times n}$
with the \emph{parameter space}
$\mathbb R^r$ is an analytic map
\[
{\cal A}: (\mathbb R^r,
\underline 0)\to ({\mathbb
C}^{n\times n},A)
\]
from a neighborhood of $\underline 0$
in $\mathbb R^r$ to ${\mathbb
C}^{n\times n}$ such that ${\cal
A}(\underline 0)=A$. (Thus, ${\cal
A}={\cal A}(x_1,\dots,x_r)$ is a
parameter matrix whose parameters
$x_1,\dots,x_r$ are real numbers.)

Let ${\cal A}$ and ${\cal B}$ be two
$\mathbb R$-deformations of
$A\in{\mathbb C}^{n\times n}$ with the
same parameter space $\mathbb R^r$. We
consider ${\cal A}$ and ${\cal B}$ as
\emph{equal} if they coincide on some
neighborhood of $\underline 0$ in
$\mathbb R^r$. We say that ${\cal A}$
and ${\cal B}$ are \emph{*congruent} if
there exists an $\mathbb R$-deformation
${\cal S}: (\mathbb R^r,\underline
0)\to ({\mathbb C}^{n\times n},I_n)$ of
$I_n$ such that
\begin{equation*}\label{kft}
{\cal
S}(\underline x)^*
{\cal A}(\underline x)
{\cal S}(\underline x)=
{\cal B}(\underline x)
\end{equation*}
for all $\underline x=(x_1,\dots, x_r)$
in a neighborhood of $\underline 0$ in
$\mathbb R^r$.

Let ${\cal A}:(\mathbb R^r,\underline
0)\to ({\mathbb C}^{n\times n},A)$ and
${\cal B}:(\mathbb R^s,\underline 0)
\to ({\mathbb C}^{n\times n},A)$ be two
$\mathbb R$-deformations of $A$. We say
that an analytic map $ \varphi:(\mathbb
R^s,\underline 0)\to (\mathbb
R^r,\underline 0)$ \emph{embeds ${\cal
B}$ into ${\cal A}$} if $ {\cal
B}(\underline y)= {\cal
A}(\varphi(\underline y))$ for all
$\underline y$ in a neighborhood of
$\underline 0\in \mathbb R^s$.

\begin{definition}\label{d}
An $\mathbb R$-deformation ${\cal A}:
(\mathbb R^r,\underline 0) \to
({\mathbb C}^{n\times n},A)$ is called
\begin{itemize}
     \item \emph{versal} if every
         $\mathbb R$-deformation of
         $A$ is *congruent to an
         $\mathbb R$-deformation of
         $A$ that is embedded into
         ${\cal A}$;

\item \emph{miniversal} if it is
    versal and there is no versal
    $\mathbb R$-deformation of $A$
    whose parameter space has a
    dimension less than $\dim
    _{\mathbb R}V$.
     \end{itemize}
\end{definition}

\begin{definition}\label{dhy}
Let ${\cal D}$ be a
$(0,\!*,\!\circ,\!\bullet)$ matrix  of
size $n\times n$. Replace
      each $(k,l)$ entry $*$ by
      $x_{kl}+iy_{kl}$,
each $(k,l)$ entry $\circ$ by
      $x_{kl}$,
and each $(k,l)$ entry $\bullet$ by
      $iy_{kl}$, in which all
      $x_{kl},y _{kl}$
      are
 independent real
parameters. Denote the obtained
parameter matrix by ${\cal D}(x\cup
y)$, in which
\[
x:=\{x_{ij}\,|\,(i,j)\in {\cal
I}_*\cup {\cal I}_{\circ}\},\qquad
y:=\{y_{ij}\,|\,(i,j)\in {\cal
I}_*\cup {\cal I}_{\bullet}\},
\]
and
\begin{equation}\label{jpx}
{\cal
I}_*,\ {\cal I}_{\circ},\ {\cal
I}_{\bullet},\  {\cal I}_0\subseteq
\{1,\dots,n\}\times \{1,\dots,n\}
\end{equation}
are the sets of indices of all stars,
circles, bullets, and zeros in ${\cal
D}$. For each $n\times n$ matrix $A$,
the $\mathbb R$-deformation $A+{\cal
D}(x\cup y)$ of $A$ with the parameter
space ${\cal D}(\mathbb C)$ (see
\eqref{kud}) is called \emph{simplest}.
\end{definition}

For example, if all entries of $\cal D$
are stars, then it defines the simplest
$\mathbb R$-deformation
\begin{equation}\label{edr}
\mathcal U(x\cup y):=A+
\begin{bmatrix}
  x_{11}+iy_{11} & \dots &
  x_{1n}+iy_{1n} \\
  \vdots & \ddots & \vdots \\
  x_{n1}+iy_{n1} & \dots &
  x_{nn}+iy_{nn}  \\
\end{bmatrix}
\end{equation}
with the parameter space \[\mathbb
C^{n\times n}=
\begin{bmatrix}
  \mathbb R+i \mathbb R & \dots &
  \mathbb R+i \mathbb R \\
  \vdots & \ddots & \vdots \\
\mathbb R+i \mathbb R & \dots &
  \mathbb R+i \mathbb R \\
\end{bmatrix}.\]
This deformation is \emph{universal} in
the sense that every $\mathbb
R$-deformation $\cal B$ of $A$ is
embedded into it (since every entry of
$\cal B$ is a complex-valued function
of real variables).

Since each square matrix is *congruent
to its canonical matrix, it suffices to
construct miniversal $\mathbb
R$-deformations of all canonical
matrices \eqref{juw}. These $\mathbb
R$-deformations are given in the
following theorem, which is another
form of Theorem \ref{teo2}.

\begin{theorem}\label{teojy}
Let $A_{\text{\rm can}}$ be a canonical
matrix \eqref{juw} for *congruence, and
let ${\cal D}={\cal D}\langle
A_{\text{can}}\rangle $ be the
$(0,\!*,\!\circ,\!\bullet)$ matrix from
Definition \ref{kkj}. Then the simplest
$\mathbb R$-deformation $A_{\text{\rm
can}} +{\cal D}({x\cup y})$ from
Definition \ref{dhy} is miniversal.
\end{theorem}

\begin{lemma}\label{los}
Theorem \ref{teojy} follows from
Theorem \ref{teo2}.
\end{lemma}

\begin{proof}
Let Theorem \ref{teo2} hold; that is,
for $A_{\text{can}}$ there exist matrix
functions ${\cal S}:(\mathbb C^{n\times
n},0_n)\to (\mathbb C^{n\times n},I_n)$
and ${\it\Phi}: (\mathbb C^{n\times
n},0_n)\to ({\cal D}(\mathbb C),0_n)$
(with ${\cal D}$ from Definition
\ref{kkj}) that are analytic functions
of the entries of $X$ and $\bar X$ such
that
\begin{equation}\label{KKT}
{\cal
S}(X)^*
(A_{\text{can}}+X) {\cal S}(X)
=A_{\text{can}}+{\it\Phi}(X)
\end{equation}
for all $X\in\mathbb C^{n\times
n}$ in a neighborhood of $0_n$.

Let ${\cal B}:(\mathbb R^s,\underline
0) \to ({\mathbb C}^{n\times
n},A_{\text{can}})$ be an arbitrary
$\mathbb R$-deformation of
$A_{\text{can}}$. Then
\[{\it\Psi}:
(\mathbb R^s,\underline 0)\to (\mathbb
C^{n\times n},0_n),\qquad
{\it\Psi}(\underline z):={\cal
B}(\underline z)-A_{\text{can}}
\] is an analytic map
and so
\[
\hat{\it\Psi}
 %%%%
: (\mathbb
R^s,\underline 0)\to ((\mathbb
R^{n\times n},\mathbb
R^{n\times n}),0_n),\qquad
\hat{\it\Psi}(\underline z):=
(\re{\it\Psi}(\underline z),
\im{\it\Psi}(\underline z))
\]
is an analytic map too since
\[
\re{\it\Psi}(\underline z)=
\frac{{\it\Psi}(\underline z)+
\overline{{\it\Psi}(\underline z)}}2,
      \qquad
\im{\it\Psi}(\underline z)=
\frac{{\it\Psi}(\underline z)-
\overline{{\it\Psi}(\underline z)}}{2i}
\]
The map $\hat{\it\Psi}$ embeds ${\cal
B}$ into the universal $\mathbb
R$-deformation \eqref{edr} since ${\cal
B}(\underline z)=\mathcal
U(\hat{\it\Psi}(\underline
z))=A_{\text{can}} +
{\it\Psi}(\underline z)$. By
\eqref{KKT},
\[
{\cal
S}({\it\Psi}(\underline z))^*
(A_{\text{can}}+
{\it\Psi}(\underline z))
{\cal S}({\it\Psi}(\underline z))
=A_{\text{can}}+
{\it\Phi}
({\it\Psi}(\underline z))
\]
and so ${\cal B}$ is *congruent to an
$\mathbb R$-deformation that is
embedded into $A_{\text{can}}+{\cal
D}(x\cup y) $, which proves that
$A_{\text{can}}+{\cal D}(x\cup y) $ is
versal. It is miniversal since by
Theorem \ref{teo2} the number
$\dim_{\mathbb R}{\cal D}({\mathbb C})$
is the smallest that can be made by
transformations \eqref{tef};
$\dim_{\mathbb R} T( A_{\text{can}})$
in \eqref{jyr} is the number of
independent reducing parameters.
\end{proof}

\section{Construction of the
transforming matrix} \label{sect4}

In this section we prove that Theorem
\ref{teo2} holds if we take any
$(0,\!*,\!\circ,\!\bullet)$ matrix
${\cal D}$ satisfying \eqref{jyr}
instead of ${\cal D}={\cal D}\langle
A_{\text{can}}\rangle $ from Definition
\ref{kkj}; we also construct the
transforming matrix ${\cal S}(X)$. For
miniversal deformations of matrices
under congruence, transforming matrices
were constructed in a similar way by
the authors in \cite[Appendix A]{f_s}.
For miniversal deformations of matrices
under similarity and of matrix pencils,
transforming matrices (in the form of
Taylor series) were constructed and
their numerous applications were given
by Garcia-Planas and Mailybaev in
\cite{gar_mai,mai2000,mai2001}.

Let us fix  an $n\times n$ complex
matrix $A$ and a
$(0,\!*\!,\circ,\!\bullet)$ matrix
$\cal D$ of the same size such that
${\mathbb C}^{\,n\times n}=T(A) + {\cal
D}({\mathbb C})$ (we do not assume that
this sum is direct).

 For each
$M=[m_{kl}]\in\mathbb
      C^{n\times n}$, we write
$\|M\|:=\sqrt{\sum |m_{kl}|^2}$ and
\begin{equation}\label{ntu}
\|M\|_{\cal
D}:=\sqrt{\sum_{(k,l)\in{\cal I}_0}
|m_{kl}|^2+
\sum_{(k,l)\in {\cal I}_{\circ}}
\im (m_{kl})^2+
\sum_{(k,l)\in {\cal I}_{\bullet}}
\re (m_{kl})^2}
\end{equation}
(see \eqref{jpx}), in which
$|m_{kl}|$, $\im (m_{kl})$, and $\re
(m_{kl})$ are the modulus, imaginary
part, and real part of $m_{kl}$. Note
that
\begin{equation*}\label{lk}
\|\alpha M+\beta N\|\le
|\alpha |\,\|M\|+|\beta |\,\|N\|,\qquad
\|MN\|\le \|M\|\,\|N\|
\end{equation*}
for all $\alpha ,\beta \in\mathbb C$
and $M,N\in\mathbb C^{n\times n}$ (see
\cite[Section 5.6]{hor_John}), and
\begin{equation}\label{fue}
\|M+N\|_{\cal D}=\|M\|_{\cal D}\qquad
\text{if }\ N\in {\cal
D}({\mathbb C}).
\end{equation}

For each $n\times n$ matrix unit
$E_{kl}$, we fix
$P_{kl},Q_{kl}\in\mathbb C^{n\times n}$
such that
\begin{equation}\label{8}
\begin{matrix}
E_{kl}+P_{kl}^*A +AP_{kl}\in {\cal
D}({\mathbb C})\\
iE_{kl}+Q_{kl}^*A +AQ_{kl}\in {\cal
D}({\mathbb C})\end{matrix}
\end{equation}
($P_{kl}$ and $Q_{kl}$ exist because
${\mathbb C}^{\,n\times n}=T(A) + {\cal
D}({\mathbb C})$). We can and will take
\begin{equation}\label{jup}
\text{$P_{kl}:=0_n$ if $(k,l)\in {\cal I}_*\cup{\cal
I}_{\circ}$
\quad   and \quad
   $Q_{kl}:=0_n$ if
$(k,l)\in {\cal
I}_*\cup{\cal I}_{\bullet}$}
\end{equation}
since then $E_{kl}\in {\cal D}({\mathbb
C})$ and $iE_{kl}\in {\cal D}({\mathbb
C})$, respectively.
 Write
\begin{equation}\label{kux}
a:=\|A\|,\qquad
b:=\max_{k,l}(\|P_{kl}\|,\|Q_{kl}\|).
\end{equation}

For each $n\times n$ matrix  $X$, we
construct a sequence of $n\times n$
matrices
\begin{equation}\label{rtg}
M_1:=X,\ M_2,\ M_3,\dots
\end{equation}
in which $M_{t+1}$ is defined by
$M_t=[m_{kl}^{(t)}]$ as follows:
\begin{equation}\label{dei}
A+M_{t+1}:=(I_n+C_t)^*(A+M_t)(I_n+C_t)
\end{equation}
where
\begin{equation}\label{drj1}
C_t:=\sum_{k,l}\left(\re (m_{kl}^{(t)})P_{kl}
+\im (m_{kl}^{(t)})Q_{kl}\right).
\end{equation}

In this section we prove the following
theorem.

\begin{theorem}\label{ttft}
Given $A\in\mathbb C^{n\times n}$  and
a $(0,\!*\!,\circ,\!\bullet)$ matrix
$\cal D$ of the same size such that
${\mathbb C}^{\,n\times n}=T(A) + {\cal
D}({\mathbb C})$.

{\rm(a)} All matrices $A+X$ that are
sufficiently close to $A$ are
simultaneously reduced by a
transformation
\[
{\cal
S}(X)^* (A+X) {\cal
S}(X),\quad\begin{matrix}
\text{${\cal S}(X)$
is an analytic function of the entries}\\
\text{of $\re X$ and $\im X$, and
${\cal
S}(0_n)=I_n$,}
\end{matrix}
\]
to matrices from $A +{\cal D}(\mathbb
C)$.

{\rm(b)} The transforming matrix ${\cal
S}(X)$ can be constructed as follows.
Fix $\varepsilon\in \mathbb R $ such
that
\begin{equation}\label{eoj}
0<\varepsilon<\min\left(
\frac{1}{b(a+1)(b+2)},\frac{1}3\right)
\quad (\text{see \eqref{kux}})
\end{equation}
and define the neighborhood
\[
U:=\{X\in\mathbb C^{n\times n}\,|\,
\|X\|<\varepsilon^5\}
\]
of $0_n$. Then for each matrix $X\in
U$, the infinite product
\begin{equation}\label{gre}
(I_{n}+C_1)(I_{n}+C_2)(I_{n}+C_3)\cdots
\quad (\text{see \eqref{drj1}})
\end{equation}
converges to a nonsingular matrix
${\cal S}(X)$, which is an analytic
function of the entries of $\re X$ and
$\im X$, and
\begin{equation}\label{msu1}
A+D:={\cal
S}(X)^*(A+X){\cal S}(X)\in A+{\cal
D}(\mathbb C), \qquad \|D\|\le\varepsilon ^3.
\end{equation}
\end{theorem}

The proof of Theorem \ref{ttft} is
based on the following two lemmas.

\begin{lemma} \label{lem2z}
Let $\varepsilon\in\mathbb R$,
$0<\varepsilon <1/3$, and let a
sequence of positive real numbers
\begin{equation*}\label{21z}
\delta_1,\
\mu_1,\
\delta_2,\
\mu_2,\
\delta_3,\
\mu_3,\ \dots
\end{equation*}
be defined by induction:
\begin{equation*}\label{22z}
\delta_1=\mu_1:=\varepsilon^5,\qquad
\delta_{i+1}
:=\varepsilon^{-1}\delta_i\mu _i,\qquad
\mu_{i+1}:=\mu_i
+\varepsilon^{-1}\delta_i.
\end{equation*}
Then
\begin{equation}\label{23z}
\delta_{i}<\min(
\varepsilon ^{2i},\mu_i),\qquad
\mu_i<\varepsilon ^{3}\qquad\text{for all
$i=1,2,\dots$}
\end{equation}
\end{lemma}

\begin{proof}
Reasoning by induction, we assume that
the inequalities \eqref{23z} hold for
$i=1,\dots,t$. Then they hold for
$i=t+1$ since
\begin{align*}
&\delta_{t+1}
=\varepsilon^{-1}\delta_t \mu_t<
\varepsilon^{-1}\varepsilon ^{2t}
\varepsilon ^{3}
= \varepsilon ^{2(t+1)},
\\
&\delta_{t+1}
=\varepsilon^{-1}\delta_t\mu _t
<\varepsilon^{-1}\delta_t<
\varepsilon^{-1}\delta_t+ \mu_t
=\mu_{t+1}
\end{align*}
and
\begin{align*}
\mu_{t+1}&=\mu_t
+\varepsilon^{-1}\delta_t=
\mu_{t-1}+\varepsilon^{-1}\delta_{t-1}+
\varepsilon^{-1}\delta_t=\cdots=
\mu_{1}+\varepsilon^{-1}(\delta_{1}
+\delta_{2}+\dots+
\delta_{t})
      \\&<
\varepsilon ^{5}+\varepsilon^{-1}
(\varepsilon ^{5}
+\varepsilon^{-1} \varepsilon ^{5}\varepsilon ^{5}
+\varepsilon ^{6}+\varepsilon ^{8}
+\varepsilon ^{10}+\cdots)\\
& =\varepsilon ^{5}+\varepsilon ^{4}
+\varepsilon ^{8}+\varepsilon ^{5}
(1+\varepsilon ^{2}+\varepsilon ^{4}+\cdots)
=\varepsilon ^{5}+\varepsilon ^{4}
+\varepsilon ^{8}+\varepsilon ^{5}/(1
-\varepsilon ^{2})\\
&<\varepsilon ^{5}+\varepsilon ^{4}
+\varepsilon ^{8}+2\varepsilon ^{5}
<3\varepsilon ^{5}+\varepsilon ^{4}
+\varepsilon ^{8}
<\varepsilon ^{4}+ \varepsilon
^{4}+\varepsilon ^{8} <3\varepsilon ^4
<\varepsilon ^{3}.\tag*{\qedhere}
\end{align*}
\end{proof}

\begin{lemma} \label{lem1}
Let $\varepsilon \in\mathbb R$ satisfy
\eqref{eoj}. Then the matrices from
\eqref{rtg} and \eqref{drj1} satisfy
\begin{equation}\label{15}
\|M_i\|<\mu_i,\quad
\|M_i\|_{\cal
D}<\delta_i,
\quad \|C_i\|\le \delta_ib
<\varepsilon^{2i-1},
\qquad
i=1,2,\dots
\end{equation}
in which $\mu_i, \delta_i, b$ are
defined in Lemma \ref{lem2z} and in
\eqref{kux}.
\end{lemma}

\begin{proof}
First we prove that for each $i$ the
third inequality in \eqref{15} follows
from the second inequality. Let
$\|M_i\|_{\cal D}<\delta_i$.

If $(k,l)\in{\cal I}_0\cup{\cal
I}_{\bullet}$, then
$|\re(m_{kl}^{(i)})|<\delta_i$ by
\eqref{ntu}. If $(k,l)\in{\cal
I}_*\cup{\cal I}_{\circ}$, then
$P_{kl}=0$ by \eqref{jup}.

If $(k,l)\in{\cal I}_0\cup{\cal
I}_{\circ}$, then
$|\im(m_{kl}^{(i)})|<\delta_i$ by
\eqref{ntu}. If $(k,l)\in{\cal
I}_*\cup{\cal I}_{\bullet}$, then
$Q_{kl}=0$ by \eqref{jup}.

Using these assertions and
\eqref{drj1}, we get
\begin{align*}\nonumber
\|C_i\|&\le \sum_{k,l}\left(|\re
(m_{kl}^{(i)})|\|P_{kl}\| +|\im
(m_{kl}^{(i)})|\|Q_{kl}\|\right)\\&\le
\sum_{k,l}(\delta_i\|P_{kl}\|
+\delta_i\|Q_{kl}\|)
=
\delta_i b.
\end{align*}
By \eqref{23z}, $\delta_i<\varepsilon
^{2i}.$ By \eqref{eoj},
$\varepsilon<1/b.$ Therefore, $\delta_i
b<\varepsilon ^{2i}\varepsilon
^{-1}=\varepsilon^{2i-1}$, which gives
the third inequality in \eqref{15}.

Let us prove the first two inequalities
in \eqref{15}. They hold for $t=1$
since $M_1=E\in U$ implies
$\|M_1\|<\varepsilon^5=\delta_1=\mu_1$.
Reasoning by induction, suppose that
they hold for $i=t$ and prove them for
$i=t+1$. Due to \eqref{dei},
\begin{equation}\label{18de}
M_{t+1}=M_t+C_t^*(A+M_t)+
(A+M_t)C_t+C_t^*(A+M_t)C_t
\end{equation}
and we have
\begin{align*}
\|M_{t+1}\|
&\le\|M_t\|+\|C_t\|(\|A\|+\|M_t\|)(2+\|C_t\|)
   \\
&<\mu_t+\delta_t
b(a+\mu_t)(2+
\delta_t b)\quad
\text{(by \eqref{15} for $i=t$)}
    \\
&<
\mu_t+\delta_t
b(a+1)(2+b)
< \mu_t+\delta_t\varepsilon^{-1}\quad
\text{(by \eqref{23z} and \eqref{eoj})}
\\&=\mu_{t+1},
\end{align*}
which gives the first inequality in
\eqref{15} for $i=t+1$.

Due to \eqref{8} and \eqref{drj1},
\begin{gather*}
M_t+C_t^*A+AC_t=\sum_{k,l} \left(\re
(m_{kl}^{(t)})E_{kl}+\im
(m_{kl}^{(t)})iE_{kl}\right)
              \\
              +
\sum_{k,l}
\left(\re(m_{kl}^{(t)})P_{kl}^*+\im
(m_{kl}^{(t)})Q_{kl}^*\right)A
+\sum_{k,l}A\left(\re(m_{kl}^{(t)})
P_{kl}+\im (m_{kl}^{(t)})Q_{kl}\right)
\\=\sum_{k,l}\left(\re(m_{kl}^{(t)})
(E_{kl}+P_{kl}^*A +AP_{kl})+ \im
(m_{kl}^{(t)})( iE_{kl}+Q_{kl}^*A
+AQ_{kl})\right)\in {\cal D}({\mathbb
C})
\end{gather*}
and we obtain the second inequality in
\eqref{15} for $i=t+1$:
\begin{align*}
\|M_{t+1}\|_{\cal
D}&=\|C_t^*M_t+M_tC_t+C_t^*(A+M_t)C_t\|_{\cal
D}\quad \text{(by \eqref{fue} and
\eqref{18de})} \\
&\le\|C_t^*M_t+M_tC_t+C_t^*(A+M_t)C_t\| \\
&\le 2\|C_t\|\|M_t\|+
\|C_t\|^2(\|A\|+\|M_t\|)
   \\&\le
2\delta_t
b\mu_t+(\delta_t
b)^2(a+\mu_t)
    \\
&\le \delta_t \mu_t b (2+b(a+1))
\quad \text{(since $\delta_t <\mu _t$ by
\eqref{23z})}
     \\
&\le \delta_t \mu_t b(2+b)(a+1)
<\delta_t\mu_t\varepsilon^{-1}
=\delta_{t+1}.\tag*{\qedhere}
\end{align*}
\end{proof}

\begin{proof}[Proof of Theorem \ref{ttft}]
For each $i$, $\|C_i\|<\varepsilon
^{2i-1}<1$ by \eqref{15}, and so
$I_n+C_i$ is nonsingular by
\cite[Corollary 5.6.16]{hor_John}. We
have
\begin{equation}\label{ool}
\|C_i\|<c_i,\qquad c_i:=\varepsilon ^{2i-2}.
\end{equation}

The series with nonnegative terms
$c_1+c_2+\cdots$ is convergent since
\[
c_1+c_2+c_3+\cdots
=1
+\varepsilon^2+\varepsilon^4+\cdots
=1 /(1-\varepsilon ^{2})
< 1/(1-3^{-2}).
\]
This implies the convergence of the
infinite product
\begin{equation}\label{ghe}
c:=(1+c_1)(1+c_2)\cdots
\end{equation}
due to \eqref{ool} and \cite[Theorem
15.14]{mark}. This also implies the
convergence of infinite product
\eqref{gre} to some nonsingular matrix
${\cal S}(X)$ due to \cite[Theorems 2
and 4]{inf}.

Let us prove that ${\cal S}(X)$ on the
neighborhood $U$ is an analytic
function of the entries of $\re X$ and
$\im X$. To this end, we first prove
that
\[{\cal
S}_t(X)=[s_{kl}^{(t)}]:=(I_{n}+C_1)(I_{n}+C_2)\cdots
(I_{n}+C_t)\to {\cal S}(X)\] uniformly on $U$ as
$t\to\infty$. Since each entry
$s_{kl}^{(t)}$ is represented in the
form
\[s_{kl}^{(1)}+(s_{kl}^{(2)}
-s_{kl}^{(1)})+(s_{kl}^{(3)}
-s_{kl}^{(2)})+\dots+(s_{kl}^{(t)}
-s_{kl}^{(t-1)}),\] it suffices to prove
that the series
\begin{equation}\label{dmi}
s_{kl}^{(1)}+(s_{kl}^{(2)}
-s_{kl}^{(1)})+(s_{kl}^{(3)}
-s_{kl}^{(2)})+\dots+(s_{kl}^{(t)}
-s_{kl}^{(t-1)})+\cdots
\end{equation}
converges uniformly on $U$. The latter
follows from the Weierstrass M-test
\cite[Theorem 7.10]{rud} since the
convergent series of positive constants
$c(c_1+c_2+c_3+\cdots)$ (see
\eqref{ghe}) is a majorant of
\eqref{dmi}:
\begin{align*}
|s_{kl}^{(1)}|<&\|{\cal
S}_{1}(X)\| =\|I_{n}+C_1\|\le
1+c_1<cc_1,\\
|s_{kl}^{(t)}-s_{kl}^{(t-1)}|<&
\|{\cal S}_{t}(X)-{\cal S}_{t-1}(X)\|
=\|{\cal S}_{t-1}(X)C_{t}\|
     \\
&\qquad\le(1+c_1)\cdots(1+c_{t-1})
c_{t}<cc_{t}.
\end{align*}

Each entry of $\re C_t$ and $\im C_t$
for every $t$  is a polynomial function
of the entries of $\re X$ and $\im X$,
which is easily proved by induction on
$t$ using \eqref{rtg}--\eqref{drj1}.
Hence, each entry of ${\cal S}_t(X)$ is
a polynomial function of the entries of
$\re X$ and $\im X$. Since ${\cal
S}_t(X)\to {\cal S}(X)$, the
Weierstrass theorem on uniformly
convergent sequences of analytic
functions \cite[Theorem 7.12]{nar}
ensures that the entries of ${\cal
S}(X)$ are analytic functions of the
entries of $\re X$ and $\im X$.

The inclusion in \eqref{msu1} holds
since by \eqref{dei} \[A+M_{t}={\cal
S}_{t-1}(X)^*(A+X){\cal S}_{t-1}(X)\to
{\cal S}(X)^*(A+X){\cal S}(X)\] and
$\|M_t\|_{\cal D} <\delta_{t}\to 0$ as
$t\to\infty$. The inequality in
\eqref{msu1} holds because $M _t\to D$
and $ \|M_t\|< \mu_t<\varepsilon ^3.$
\end{proof}

\section{Proof of Theorem
\ref{teo2}}\label{kku}

In this section we prove that ${\mathbb
C}^{\,n\times n}=T(A_{\text{can}})
\oplus_{\mathbb R} {\cal D}({\mathbb
C}) $ for the
$(0,\!*,\!\circ,\!\bullet)$ matrix
${\cal D}$ from Definition \ref{kkj},
which ensures Theorem \ref{teo2} due to
Theorem \ref{ttft}(a).

For each $M\in {\mathbb C}^{m\times m}$
and $N\in {\mathbb C}^{n\times n}$,
define the real vector space of matrix
pairs
\begin{equation*}\label{neh1}
 T(M,N):=\{( S^*M+NR,\:
R^*N+MS)\,|\,S\in
 {\mathbb C}^{m\times n},\ R\in
 {\mathbb C}^{n\times m}\}.
\end{equation*}

\begin{lemma}\label{thekd}
Let
\begin{equation}\label{ftu}
A=A_1\oplus\dots\oplus A_t\in
\mathbb C^{n\times n}
\end{equation}
be a block-diagonal matrix in which
every $A_i$ is $n_i\times n_i$. Let
${\cal D}=[{\cal D}_{ij}]$ be a
$(0,\!*,\!\circ,\!\bullet)$ matrix of
the same size that is partitioned into
blocks conformally to the partition of
$A$. Then ${\mathbb C}^{\,n\times
n}=T(A)\oplus_{\mathbb R} {\cal
D}({\mathbb C}) $ if and only if
\begin{itemize}
  \item[\rm(i)] for each
      $i=1,\dots,t$, every
      $n_i\times n_i$ matrix can be
      reduced to exactly one matrix
      from ${\cal D}_{ii}(\mathbb
      C)$ by adding matrices from
      $T(A_i)$, and

  \item[\rm(ii)] for each
      $i,j=1,\dots,t$, $i<j$, every
      pair of $n_j\times n_i$ and
      $n_i\times n_j$ matrices can
      be reduced to exactly one
      matrix pair from $({\cal
      D}_{ji}(\mathbb C),{\cal
      D}_{ij}(\mathbb C))$ by
      adding matrix pairs from
      $T(A_i,A_j)$.
\end{itemize}
\end{lemma}

\begin{proof}
Clearly, ${\mathbb C}^{\,n\times
n}=T(A)\oplus_{\mathbb R} {\cal
D}({\mathbb C}) $ if and only if for
each $C\in{\mathbb C}^{n\times n}$ the
set $C+T(A)$ contains exactly one
matrix $D$ from ${\cal D}(\mathbb C)$;
i.e., there is exactly one
\begin{equation}\label{kid}
D=C+S^*A+AS\in{\cal
D}(\mathbb C)\qquad
\text{with
$S\in{\mathbb
C}^{n\times n}$.}
\end{equation}
Partition $D,\ C$, and $S$ into $t^2$
blocks conformally to the partition
\eqref{ftu} of $A$. By \eqref{kid}, for
each $i$ we have
$D_{ii}=C_{ii}+S_{ii}^*A_{i}
+A_{i}S_{ii}$, and for each $i,j$ such
that $i<j$ we have
\begin{equation*}\label{mht}
\begin{bmatrix}
D_{ii}&D_{ij}
 \\ D_{ji}&D_{jj}
\end{bmatrix}
=
\begin{bmatrix}
C_{ii}&C_{ij}
 \\ C_{ji}&C_{jj}
\end{bmatrix}
+ \begin{bmatrix}
S_{ii}^*&S_{ji}^*
 \\ S_{ij}^*&S_{jj}^*
\end{bmatrix}
\begin{bmatrix}
A_i&0
 \\ 0& A_j
\end{bmatrix}
+
\begin{bmatrix}
A_i&0
 \\ 0& A_j
\end{bmatrix}
\begin{bmatrix}
S_{ii}&S_{ij}
 \\ S_{ji}&S_{jj}
\end{bmatrix}.
\end{equation*}
Thus, \eqref{kid} can be rewritten in
the form
\begin{align*}%\label{djh}
D_{ii}&=C_{ii}
+S_{ii}^*A_i+A_iS_{ii}\in{\cal
D}_{ii}(\mathbb
C)\\%\label{djhh}
(D_{ji},D_{ij})&= (C_{ji},C_{ij})
+(S_{ij}^*A_i+A_jS_{ji},\:
S_{ji}^*A_j+A_iS_{ij}) \in {\cal
D}_{ji}(\mathbb C)\oplus {\cal
D}_{ij}(\mathbb C)
\end{align*}
for all $i=1,\dots,t$ and
$j=i+1,\dots,t$.
\end{proof}

\begin{corollary}\label{cor}
It suffices to prove Theorem \ref{teo2}
for those $A_{\text{\rm can}}$ in
\eqref{juw} that have at most two
direct summands; i.e., it suffices to
prove that
\begin{itemize}
  \item[\rm(i)] for each
      $(0,\!*,\!,\circ,\!\bullet)$
      matrix ${\cal D}:={\cal
      D}\langle M\rangle$ from
      Definition \ref{kkj}(i)
      $($i.e., $M$ is
      $H_{2m}(\lambda)$, $\mu
      \Delta_n$, or $J_n(0))$,
      every matrix of the same size
      as $\cal D$ can be reduced to
      exactly one matrix from
      ${\cal D}(\mathbb C)$ by
      adding matrices from $T(M)$,
      and

  \item[\rm(ii)] for each pair of
      $(0,\!*,\!,\circ,\!\bullet)$
      matrices $({\cal D}_1,{\cal
      D}_2):=({\cal D}\langle
      M\rangle ,{\cal D}\langle
      N\rangle )$ from Definition
      \ref{kkj}(ii), every matrix
      pair of the
 same size as $({\cal D}_1,{\cal
      D}_2)$
can be reduced to exactly one
      matrix pair from $({\cal
      D}_1(\mathbb C),{\cal
      D}_2(\mathbb C))$ by adding
      matrix pairs from $T( M,N)$.
\end{itemize}
\end{corollary}

In the rest of the paper we prove that
the assertions (i) and (ii) of
Corollary \ref{cor} hold.

\subsection{Diagonal blocks of
${\cal D}\langle A_{\text{\rm
can}}\rangle $}

In this section we prove the assertion
(i) of Corollary \ref{cor}.

\subsubsection{Blocks ${\cal
D}\langle H_{2n}(\lambda)\rangle $ with
$|\lambda|>1$} \label{sub2}

According to Corollary \ref{cor}(i), we
need to prove that each $2n\times 2n$
matrix $A=[A_{ij}]_{i,j=1}^2$ can be
reduced to exactly one matrix of the
form \eqref{KEV} by adding
\begin{multline*}\label{moh}
\begin{bmatrix}
S_{11}^*&S_{21}^*
 \\ S_{12}^*&S_{22}^*
\end{bmatrix}
\begin{bmatrix}
0&I_n
 \\ J_n(\lambda)&0
\end{bmatrix}
+\begin{bmatrix} 0&I_n
 \\ J_n(\lambda)&0
\end{bmatrix}
\begin{bmatrix}
S_{11}&S_{12}
 \\ S_{21}&S_{22}
\end{bmatrix}
    \\=
\begin{bmatrix}
S_{21}^* J_n(\lambda)+S_{21}&
S_{11}^*+S_{22}\\
S_{22}^* J_n(\lambda)+J_n(\lambda)
S_{11}& S_{12}^*+J_n(\lambda)S_{12}
\end{bmatrix}
\end{multline*}
in which $S=[S_{ij}]_{i,j=1}^2$ is an
arbitrary $2n\times 2n$ matrix. Taking
$S_{22}=-A_{12}$ and all other
$S_{ij}=0$, we obtain a new matrix $A$
with $A_{12}=0$. To preserve $A_{12}$,
we must hereafter take $S$ with
$S_{11}^*+S_{22}=0$. Therefore, we can
add $S_{21}^* J_n(\lambda)+S_{21}$ to
(the new) $A_{11}$,
$S_{12}^*+J_n(\lambda)S_{12}$ to
$A_{22}$, and $-S_{11}
J_n(\lambda)+J_n(\lambda) S_{11}$ to
$A_{21}$. Using these additions we can
reduce $A$ to the form \eqref{KEV} due
to the following 3 lemmas.

\begin{lemma}\label{lem2aaaa}
By adding $SJ_n(\lambda)+S^*$, in which
$\lambda$ is a fixed complex number,
$|\lambda |\ne 1$, and $S$ is
arbitrary, we can reduce each $n\times
n$ matrix to $0$.
\end{lemma}

\begin{proof}
Let $A=[a_{ij}]$ be an arbitrary
$n\times n$ matrix. We will reduce it
along its \emph{skew diagonals}
starting from the upper left hand
corner; that is, along
\begin{equation*}\label{ly}
a_{11},\
(a_{21},a_{12}),\
(a_{31},a_{22},a_{13}),\
\dots,\ a_{nn},
\end{equation*}
by adding $\varDelta
A:=SJ_n(\lambda)+S^*$ in which
$S=[s_{ij}]$ is any $n\times n$ matrix.
For instance, if $n=4$ then $\varDelta
A$ is
\[
\begin{bmatrix}
 \lambda s_{11}+0+\bar s_{11} &
 \lambda s_{12}+s_{11}+\bar s_{21} &
 \lambda s_{13}+s_{12}+\bar s_{31} &
 \lambda s_{14}+s_{13}+\bar s_{41}
            \\
\lambda s_{21}+0+\bar s_{12}
&
 \lambda s_{22}+s_{21}+\bar s_{22} &
 \lambda s_{23}+s_{22}+\bar s_{32} &
 \lambda s_{24}+s_{23}+\bar s_{42}
             \\
\lambda s_{31}+0+\bar s_{13}
&
\lambda s_{32}+s_{31}+\bar s_{23} &
\lambda  s_{33}+s_{32}+\bar s_{33} &
\lambda s_{34}+s_{33}+\bar s_{43}
  \\
\lambda s_{41}+0+\bar s_{14}
&
\lambda  s_{42}+s_{41}+\bar s_{24} &
\lambda  s_{43}+s_{42}+\bar s_{34} &
\lambda   s_{44}+s_{43}+\bar s_{44}
   \end{bmatrix}.
\]

We reduce $A$ to $0$ by induction.
Assume that the first $t-1$ skew
diagonals of $A$ are zero. To preserve
them, we take the first $t-1$ skew
diagonals of $S$ equalling zero. If the
$t{\text{\rm th}}$ skew diagonal of $S$
is $(x_1,\dots,x_r)$, then we can add
\begin{equation}\label{mugh}
(\lambda x_{1}+\bar x_r,\
\lambda
x_{2}+\bar x_{r-1},\
\lambda
x_{3}+\bar x_{r-2},\ \dots,\ \lambda
x_{r}+\bar x_1)
\end{equation}
to the $t{\text{\rm th}}$ skew diagonal
of $A$. Let us show that each vector
$(c_1,\dots,c_r)\in\mathbb C^{r}$ is
represented in the form \eqref{mugh};
that is, the corresponding system of
linear equations
\begin{equation}\label{kes}
\lambda
x_{1}+\bar x_r=c_1,\ \dots,\ \lambda x_j+
\bar x_{r-j+1}=c_j,\
\dots,\  \lambda
x_{r}+\bar x_1 =c_r
\end{equation}
has a solution. This is clear if
$\lambda =0$. Suppose that $\lambda \ne
0$.

By \eqref{kes}, $x_j=\lambda
^{-1}(c_j-\bar x_{r-j+1})$. Replace $j$
by $r-j+1$:
\begin{equation}\label{fsw}
x_{r-j+1}=\lambda
^{-1}(c_{r-j+1}-\bar x_j).
\end{equation}

Consider only the case $r=2k+1$ (the
case $r=2k$ is considered analogously).
Substituting \eqref{fsw} into the first
$k+1$ equations of \eqref{kes}, we
obtain
\begin{equation*}\label{jdi}
 \lambda x_j+\bar\lambda ^{-1}(\bar c_{r-j+1}-
x_{j})=(\lambda -\bar\lambda ^{-1})x_j+
\bar\lambda ^{-1}\bar c_{r-j+1}=c_j,\qquad j=1,\dots,k+1.
 \end{equation*}
Since $|\lambda|\ne 1$, $\lambda
-\bar\lambda ^{-1}\ne 0$ and we have
\begin{equation}\label{j,f}
x_j
=\frac{c_j-\bar\lambda ^{-1}\bar c_{r-j+1}}
{\lambda
-\bar\lambda ^{-1}}=
\frac{\bar\lambda c_j-\bar c_{r-j+1}}
{\lambda\bar\lambda
-1}
,\qquad j=1,\dots,k+1.
 \end{equation}
The equalities \eqref{fsw} and
\eqref{j,f} give the solution of
\eqref{kes}.
\end{proof}

\begin{lemma}\label{lem2a}
By adding $J_n(\lambda)R+R^*$, in which
$\lambda$ is a fixed complex number,
$|\lambda |\ne 1$, and $R$ is
arbitrary, we can reduce each $n\times
n$ matrix to $0$.
\end{lemma}

\begin{proof}
By Lemma \ref{lem2aaaa}, for each
$n\times n$ matrix $B$ there exists $S$
such that $B+SJ_n(\bar\lambda)+S^*=0$.
Then $B^*+J_n(\bar\lambda)^*S^*+S=0. $
Since
\[
ZJ_n(\bar\lambda)^*Z=J_n(\lambda),\qquad
Z:= \begin{bmatrix}
 0&&1
 \\ &\udots&\\
 1&&0
 \end{bmatrix},
\]
we have
\[
ZB^*Z+J_n(\lambda)(ZSZ)^*+ZSZ=0.
\]
This implies Lemma \ref{lem2a} since
$ZB^*Z$ is arbitrary.
\end{proof}

\begin{lemma}[{\cite[Lemma 5.5]{f_s}}]\label{lem3a}
By adding $S J_n(\lambda)-J_n(\lambda)
S$, we can reduce each $n\times n$
matrix to exactly one matrix of the
form $0^{\swarrow}$.
\end{lemma}

\begin{proof}
Let $A=[a_{ij}]$ be an arbitrary
$n\times n$ matrix. Adding
\begin{multline*}%\label{}
SJ_n(\lambda)-J_n(\lambda)S
=SJ_n(0)-J_n(0)S \\=
\begin{bmatrix}
 s_{21}-0&s_{22}-s_{11}&s_{23}-s_{12}
 &\dots&s_{2n}-s_{1,n-1}
 \\ \hdotsfor{5}\\
 s_{n-1,1}-0&s_{n-1,2}-s_{n-2,1}
 &s_{n-1,3}-s_{n-2,2}
 &\dots&s_{n-1,n}-s_{n-2,n-1}\\
 s_{n1}-0&s_{n2}-s_{n-1,1}&s_{n3}-s_{n-1,2}
 &\dots&s_{nn}-s_{n-1,n-1}
 \\ 0-0&0-s_{n1}&0-s_{n2}&\dots&0-s_{n,n-1}
 \end{bmatrix}
\end{multline*}
we reduce $A$ to the form
$0^{\swarrow}$ along the diagonals
\begin{equation*}\label{lyj}
a_{n1},\
(a_{n-1,1},a_{n2}),\
(a_{n-2,1},a_{n-1,2},
a_{n3}),\ \dots,\ a_{1n}.\tag*{\qedhere}
\end{equation*}
\end{proof}

\subsubsection{Blocks ${\cal D}
\langle \mu \Delta_n\rangle $ with
$|\mu|=1$} \label{sub6}

According to Corollary \ref{cor}(i), we
need to prove that each $n\times n$
matrix $A=[a_{ij}]$ can be reduced to
exactly one matrix of the form $
0^{\sespoon}$ if $\mu\notin\mathbb R$
or $0^{\sefilledspoon}$ if $\mu\notin i
\mathbb R$ by adding
\begin{equation}\label{lim}
\varDelta A:=\mu
(S^*\Delta_n+\Delta_n S)
\end{equation}
in which $S=[s_{ij}]$ is any $n\times
n$ matrix.

For example, if $n=4$, then $\varDelta
A$ is
\[
\mu\left[
\begin{array}{llll}
\bar s_{41}+s_{41}+i(0+0)
&\bar s_{31}+s_{42}+
i(\bar s_{41}+0)
&\dots
&\bar s_{11}+s_{44}+
i(\bar s_{21}+0)
\\
\bar s_{42}+s_{31}+i(0+s_{41})
&\bar s_{32}+s_{32}+
i(\bar s_{42}+s_{42})
&\dots
&\bar s_{12}+s_{34}+
i(\bar s_{22}+s_{44})
\\
\bar s_{43}+s_{21}+i(0+s_{31})
&\bar s_{33}+s_{22}+
i(\bar s_{43}+s_{32})
&\dots
&\bar s_{13}+s_{24}+
i(\bar s_{23}+s_{34})
\\
\bar s_{44}+s_{11}+i(0+s_{21})
&\bar s_{34}+s_{12}+
i(\bar s_{44}+s_{22})
&\dots
&\bar s_{14}+s_{14}+
i(\bar s_{24}+s_{24})
\end{array}
\right]
\]

Write
\begin{equation}\label{feq}
s_{n+1,j}:=0\qquad \text{for $j=1,\dots,n$}
\end{equation}
and define $\delta _{ij}$ via
$\varDelta A=\mu[\delta _{ij}]$. Then
\begin{equation}\label{dsw}
\delta _{ij}=\bar s_{n+1-j,i}+s_{n+1-i,j}
+i(\bar s_{n+2-j,i}+s_{n+2-i,j}).
\end{equation}

\emph{Step 1: Let us prove that
\begin{equation}\label{dfg}
\exists S:\quad A+\varDelta A \text{ is a
diagonal matrix.}
\end{equation}}
Let $A=\mu[a_{ij}]$. We need to prove
that the system of equations
\begin{equation}\label{jem}
\delta_{ij}=-a_{ij},\qquad
i,j=1,\dots,n,\quad i\ne j
\end{equation}
with unknowns $s_{ij}$ is consistent
for all $a_{ij}$.

Since
\begin{equation*}\label{hiw}
\bar\delta _{ji}=\bar s_{n+1-j,i}+s_{n+1-i,j}
-i(\bar s_{n+2-j,i}+s_{n+2-i,j})=-\bar a_{ji}
\end{equation*}
we have
\begin{gather*}
(\delta _{ij}+\bar \delta _{ji})/2=
\bar s_{n+1-j,i}+s_{n+1-i,j}
=(-a_{ij}-\bar a_{ji})/2
\\
(\delta _{ij}-\bar \delta _{ji})/(2i)=
\bar s_{n+2-j,i}+s_{n+2-i,j}
=(-a_{ij}+\bar a_{ji})/(2i)
\end{gather*}
Thus, the system of equations
\eqref{jem} is equivalent to the system
\begin{equation}\label{vik}
\begin{matrix}
\bar
s_{n+1-j,i}+s_{n+1-i,j}=b_{ij}\\
\bar
s_{n+2-j,i}+s_{n+2-i,j}=c_{ij}\\
\end{matrix}
\qquad
i,j=1,\dots,n,\quad i<j
\end{equation}
in which $ b_{ij}:=(-a_{ij}-\bar
a_{ji})/2$ and $c_{ij}:=(-a_{ij}+\bar
a_{ji})/(2i).$

For $k,l=1,\dots,n$, write
\begin{equation}\label{zir}
u_{kl}:=
  \begin{cases}
    -s_{kl} & \text{if } k+l\ge n+2 \\
    \bar s_{kl} & \text{if } k+l\le n+1
  \end{cases}
\end{equation}
Then the system \eqref{vik} takes the
form
\begin{equation}\label{vikq}
\begin{matrix}
u_{n+1-j,i}-u_{n+1-i,j}=b_{ij}\\
u_{n+2-j,i}-u_{n+2-i,j}=c_{ij}\\
\end{matrix}
\qquad
i,j=1,\dots,n,\quad i<j.
\end{equation}
This system is partitioned into
subsystems with unknowns $u_{\alpha
,\beta}$, $\alpha -\beta
=\text{const}$. Each of these
subsystems has the form
\begin{equation}\label{ddo}
\begin{aligned}
&u_{1,p+1}=u_{n-p,n}+b_{p+1,n}
    &&
u_{2,p+2}=u_{n-p,n}+c_{p+2,n}
    \\
&u_{2,p+2}=u_{n-p-1,n-1}+b_{p+2,n-1}
 \quad&&
u_{3,p+3}=u_{n-p-1,n-1}+c_{p+3,n-1}
  \\
&u_{3,p+3}=u_{n-p-2,n-2}+b_{p+3,n-2}
  &&
u_{4,p+4}=u_{n-p-2,n-2}+c_{p+4,n-2}
  \\
&\dots\dots\dots\dots\dots\dots\dots\dots\dots
    &&
\dots\dots\dots\dots\dots\dots\dots\dots\dots
\end{aligned}
\end{equation}
given by $p\in\{0,1,\dots,n-2\}$, or
\begin{equation}\label{ddo1}
\begin{aligned}
&u_{p+1,1}=u_{n,n-p}+b_{1,n-p}
&&
u_{p+2,2}=u_{n,n-p}+c_{2,n-p}
\\
&u_{p+2,2}=u_{n-1,n-p-1}+b_{2,n-p-1}
\quad&&
u_{p+3,3}=u_{n-1,n-p-1}+c_{3,n-p-1}
\\
&u_{p+3,3}=u_{n-2,n-p-2}+b_{3,n-p-2}
&&
u_{p+4,4}=u_{n-2,n-p-2}+c_{4,n-p-2}
\\
&\dots\dots\dots\dots\dots\dots\dots\dots\dots
    &&
\dots\dots\dots\dots\dots\dots\dots\dots\dots
\end{aligned}
\end{equation}
given by $p\in\{1,2,\dots,n-2\}$. In
each of subsystems, all equations have
the form
$u_{\dots,i}=u_{\dots,j}+\cdots$ in
which $i<j$ (see \eqref{vik});
therefore, the set of unknowns in the
left-hand sides of equations does not
intersect with the set of unknowns in
their right-hand sides. Thus, all
subsystems are consistent. This proves
\eqref{dfg}.
\medskip

\emph{Step 2: Let us prove that for
each diagonal matrix $A$
\begin{equation}\label{dfg1}
\exists S:\quad A+\varDelta A
\text{ has the form $ 0^{\sespoon}$
if $\mu\notin\mathbb R$ or
$0^{\sefilledspoon}$ if $\mu\notin i\mathbb
R$}
\end{equation}
in which $\varDelta A$ is defined in
\eqref{lim}.}

Since $A$, $0^{\sespoon}$, and
$0^{\sefilledspoon}$ are diagonal
matrices, the matrix $\varDelta A$ must
be diagonal too. Thus, the entries of
$S$ must satisfy the system \eqref{jem}
with $a_{ij}=0$. Reasoning as in Step
1, we obtain the system \eqref{vikq}
with $b_{kl}=c_{kl}=0$, which is
partitioned into subsystems \eqref{ddo}
of the form
\begin{equation}\label{fub}
u_{1,p+1}=u_{2,p+2}=\dots=u_{n-p,n}
\qquad(p\ge 0)
\end{equation}
and subsystems \eqref{ddo1} of the form
\begin{equation}\label{fub1}
u_{p+1,1}=u_{p+2,2}=\dots=u_{n,n-p}
=u_{n+1,n-p+1}=0
\quad(p\ge 1;\text{ see \eqref{feq}).}
\end{equation}
By \eqref{zir} and \eqref{fub1}, $S$ is
upper triangular. By \eqref{fub},
\[
\bar s_{1,p+1}=\dots=\bar s_{z,p+z}
=-s_{z+1,p+z+1}=
\dots=
-s_{n-p,n}
\]
in which $z$ is the integer part of
$(n+1-p)/2$ and $p=0,1,\dots,n-2$.

Let $n=2m$ or $n=2m+1$. By \eqref{dsw},
the first $m$ entries of the main
diagonal of $\mu^{-1}\varDelta A$ are
\[
\begin{matrix}
\delta _{11}=\bar
s_{n1}+s_{n1}\\\delta _{22}=
\bar s_{n-1,2}+s_{n-1,2}
+i(\bar s_{n2}+s_{n2})\\
\hdotsfor{1}\\ \delta _{mm}= \bar
s_{n+1-m,m}+s_{n+1-m,m} +i(\bar
s_{n+2-m,m}+s_{n+2-m,m})
\end{matrix}
\]
 They are zero and so we cannot change
the first $m$ diagonal entries of $A$.

The last $m$  entries of the main
diagonal of $\mu^{-1}\varDelta A$ are
\[
\begin{matrix}
\delta _{n-m+1,n-m+1}=\bar
s_{m,n-m+1}+s_{m,n-m+1} +i(\bar
s_{m+1,n-m+1}+s_{m+1,n-m+1})
\\\hdotsfor{1}\\\delta _{n-1,n-1}=
\bar s_{2,n-1}+s_{2,n-1} +i(\bar
s_{3,n-1}+s_{3,n-1})\\ \delta _{nn}= \bar
s_{1n}+s_{1n} +i(\bar s_{2n}+s_{2n})
\end{matrix}
\]
They are arbitrary and we make the last
$m$  entries of the main diagonal of
$A$ equal to zero. This proves
\eqref{dfg1} for $n=2m$.

Let $n=2m+1$. Since $s_{m+2,m+1}=0$,
the $(m+1)$st entry of
$\mu^{-1}\varDelta A$ is
\[
\delta _{m+1,m+1}=\bar s_{m+1,m+1}+
s_{m+1,m+1}
\]
which is an arbitrary real number.
Thus, we can add $\mu r$ with an
arbitrary $r\in\mathbb R$ to the
$(m+1)$st diagonal entry of $A$.  This
proves \eqref{dfg1} for $n=2m+1$.

\subsubsection{Blocks ${\cal
D}\langle J_n(0)\rangle $} \label{sub1}

According to Corollary \ref{cor}(i), we
need to prove that each $n\times n$
matrix $A$ can be reduced to exactly
one matrix of the form $0^{\swvdash}$
by adding
\begin{multline}\label{ed2}
\varDelta A:=S^* J_n(0)+J_n(0)S
    \\=
\begin{bmatrix}
0+s_{21}&\bar s_{11}+s_{22}&\bar
s_{21}+s_{23} &\dots& \bar
s_{n-1,1}+s_{2n}
  \\
0+s_{31}&\bar s_{12}+s_{32}&\bar
s_{22}+s_{33} &\dots& \bar
s_{n-1,2}+s_{3n}
  \\
\hdotsfor{5}
  \\
0+s_{n1}&\bar s_{1,n-1}+s_{n2}&\bar
s_{2,n-1}+s_{n3} &\dots& \bar
s_{n-1,n-1}+s_{nn}
  \\
0+0&\bar s_{1n}+0&\bar s_{2n}+0 &\dots&
\bar s_{n-1,n}+0
\end{bmatrix}
\end{multline}
in which $S=[s_{ij}]$ is any $n\times
n$ matrix. Since
\begin{equation*}\label{uyhk}
\varDelta
A=[b_{ij}],\qquad
b_{ij}:=\bar
s_{j-1,i}+s_{i+1,j}
\qquad
(s_{0i}:=0,\quad
s_{n+1,j}:=0),
\end{equation*}
all entries of $\varDelta A$ have the
form $\bar s_{kl}+s_{l+1,k+1}$. Denote
by $\sim$ the equivalence relation on
$\{1,\dots,n\}\times \{1,\dots,n\}$
being the transitive and symmetric
closure of $(k,l)\sim(l+1,k+1)$.

Decompose $\varDelta A$ into the sum of
matrices
\begin{equation}\label{lip}
\varDelta
A=B_{11}+B_{12}
+\dots+B_{1n}+B_{21}+B_{31}+\dots+B_{n1}
\end{equation}
that correspond to the
equivalence classes and are defined as
follows:
\begin{itemize}
  \item Each $B_{1j}$
      $(j=1,2,\dots,n)$ is obtained
      from $\varDelta A$ by
      replacing by $0$ all its
      entries except for
\begin{equation}\label{hter}
\bar
s_{1j}+s_{j+1,2},\
\bar
s_{j+1,2}+s_{3,j+2},\
\bar
s_{3,j+2}+s_{j+3,4},\ \dots
\end{equation}

  \item Each $B_{i1}$
      $(i=2,3,\dots,n)$ is obtained
from $\varDelta A$ by replacing by
$0$ all its entries except for
\begin{equation}\label{hter1}
0+s_{i1},\
\bar
s_{i1}+s_{2,i+1}, \
\bar
s_{2,i+1}+s_{i+2,3}, \
\bar
s_{i+2,3}+s_{4,i+3}, \
\bar
s_{4,i+3}+s_{i+4,5},\ \dots
\end{equation}
\end{itemize}
The index pairs in \eqref{hter} and in
\eqref{hter1} are equivalent:
\begin{align*}
&(1,j)\sim(j+1,2)\sim
(3,j+2)\sim(j+3,4)\sim\cdots
  \\
&(i,1)\sim(2,i+1) \sim
(i+2,3)\sim
(4,i+3)\sim
(i+4,5)\sim\cdots
\end{align*}
We call the entries in \eqref{hter} and
\eqref{hter1} the \emph{main entries}
of $B_{1j}$ and $B_{i1}$ ($i>1$). The
summands $B_{11},\dots,B_{1n},B_{21},
\dots,B_{n1}$ in \eqref{lip} have no
common main entries, and so we can add
to $A$ each of these matrices
separately.

The members of the sequence
\eqref{hter} are independent: an
arbitrary sequence of complex numbers
can be represented in the form
\eqref{hter}. The members of
\eqref{hter1} are dependent only if the
last entry in this sequence has the
form $\bar s_{kn}+0$ (see \eqref{ed2}),
in which case $k$ is even, i.e.
$(k,n)=(2p,i+2p-1)$ for some $p$, and
so $i=n+1-2p$ in \eqref{hter1}. Thus
the following sequences \eqref{hter1}
are dependent:
\[
\begin{matrix}
0+s_{n-1,1},\ \bar s_{n-1,1}+s_{2n}, \
\bar s_{2n}+0
   \\
0+s_{n-3,1},\ \bar s_{n-3,1}+s_{2,n-2},
\ \bar s_{2,n-2}+s_{n-1,3}, \ \bar
s_{n-1,3}+s_{4n}, \ \bar s_{4n}+0\\
\hdotsfor{1}
\end{matrix}
\]
One of the main entries of each of the
matrices $B_{n-1,1}$, $B_{n-3,1}$,
$B_{n-5,1},\ \dots$ is expressed
through the other main entries of this
matrix, which are arbitrary. The main
entries of the other matrices $B_{i1}$
and $B_{1j}$ are arbitrary. Adding
$B_{i1}$ and $B_{1j}$, we reduce $A$ to
the form $0^{\swvdash}$.

\subsection{Off-diagonal blocks of ${\cal
D}\langle A_{\text{\rm can}}\rangle $
that correspond to summands of
$A_{\text{\rm can}}$ of the same type}

In this section we prove the assertion
(ii) of Corollary \ref{cor} for $M$ and
$N$ of the same type.

\subsubsection{Pairs of blocks ${\cal
D}\langle H_{2m}(\lambda),\,
H_{2n}(\mu)\rangle $ with $|\lambda
|,|\mu |>1$} \label{sub3}

According to Corollary \ref{cor}(ii),
we need to prove that each pair $(B,A)$
of $2n\times 2m$ and $2m\times 2n$
matrices can be reduced to exactly one
pair of the form \eqref{lsiu1} by
adding
\[
(S^*
H_{2m}(\lambda)+ H_{2n}(\mu)
R,\:R^*H_{2n}(\mu)
+H_{2m}(\lambda)S),\quad S\in
 {\mathbb C}^{2m\times 2n},\ R\in
 {\mathbb C}^{2n\times 2m}.
\]

Putting $R=0$ and
$S=-H_{2m}(\lambda)^{-1}A$, we reduce
$A$ to $0$. To preserve $A=0$, we must
hereafter take $S$ and $R$ such that
$R^*H_{2n}(\mu) +H_{2m}(\lambda)S=0$;
that is,
\[
S=-H_{2m}(\lambda)^{-1}
R^*H_{2n}(\mu)
\]
and so we can add to $B$ matrices of
the form
\[
\varDelta B:= -H_{2n}(\mu)^*
R
H_{2m}(\lambda)^{-*}H_{2m}(\lambda)+
H_{2n}(\mu) R
\]
in which $H_{2m}(\lambda)^{-*}:=
(H_{2m}(\lambda)^{-1})^*$.

 Write $
P:=-H_{2n}(\mu)^* R,$ then
$R=-H_{2n}(\mu)^{-*} P$ and
\begin{equation}\label{kif}
\varDelta B=
 P
\begin{bmatrix}
J_m(\lambda)&0\\
0&J_m(\bar{\lambda})^{-T}
\end{bmatrix}
- \begin{bmatrix}
J_n(\bar{\mu})^{-T}&0\\0&J_n(\mu)
\end{bmatrix} P.
\end{equation}

Partition $B$, $\varDelta B$, and $P$
into $n\times m$ blocks:
\[
B=\begin{bmatrix}
B_{11}&B_{12}\\
B_{21}& B_{22}
\end{bmatrix},\qquad
\varDelta
B=\begin{bmatrix}
\varDelta B_{11}&\varDelta
B_{12}\\ \varDelta
B_{21}&\varDelta B_{22}
\end{bmatrix},\qquad
P=
\begin{bmatrix}
X&Y\\Z&T
\end{bmatrix}.
\]
By \eqref{kif},
\begin{align*}
\varDelta
B_{11}&=XJ_m(\lambda)
-J_n(\bar{\mu})^{-T}X
  &
\varDelta
B_{12}&=YJ_m(\bar{\lambda})^{-T}
-J_n(\bar{\mu})^{-T}Y
 \\
\varDelta B_{21}&=
ZJ_m(\lambda)-J_n(\mu)Z
 &
\varDelta B_{22}&=
TJ_m(\bar{\lambda})^{-T}-J_n(\mu)T
\end{align*}

Thus, we can reduce each block $B_{ij}$
separately by adding $\varDelta
B_{ij}$.
\medskip

(i) Fist we reduce $B_{11}$ by adding
$\varDelta
B_{11}=XJ_m(\lambda)-J_n(\bar{\mu})^{-T}X$.
Since $|\lambda |>1$ and $|\mu|>1$, the
matrices $J_m(\lambda)$ and
$J_n(\bar{\mu})^{-T}$ have no common
eigenvalues, and so $\varDelta B_{11}$
is an arbitrary matrix. We make
$B_{11}=0$.
\medskip

(ii) Let us reduce $B_{12}$ by adding
$\varDelta B_{12}
=YJ_m(\bar{\lambda})^{-T}-J_n(\bar{\mu})^{-T}Y$.
If ${\lambda}\ne {\mu}$, then
$\varDelta B_{12}$ is arbitrary; we
make $B_{12}=0$. Let ${\lambda}={\mu}$.
Write $F:=J_n(0)$. Since
\[
J_n(\bar\lambda )^{-1}=(\bar\lambda I_n+F)^{-1}=
\bar\lambda ^{-1}I_n-\bar\lambda ^{-2}F+\bar\lambda ^{-3}F^2-
\cdots,
\]
we have
\begin{multline*}%\label{kif}
\varDelta B_{12}
=Y(J_m(\bar{\lambda})^{-T}-\bar{\lambda}^{-1}
I_m)
-(J_n(\bar{\lambda})^{-T}-\bar{\lambda}^{-1}
I_n)Y
  \\=
-\bar{\lambda}^{-2}\begin{bmatrix}
y_{12}&\dots&y_{1m}&0 \\
y_{22}&\dots&y_{2m}&0 \\
y_{32}&\dots&y_{3m}&0 \\
\hdotsfor{4}
\end{bmatrix}
    +\bar{\lambda}^{-2}
\begin{bmatrix}
0&\dots&0 \\
y_{11}&\dots&y_{1m} \\
y_{21}&\dots&y_{2m} \\
\hdotsfor{3}\\
\end{bmatrix}
+\cdots
\end{multline*}
We reduce $B_{12}$ to the form
$0^{\nearrow}$ along its diagonals
starting from the upper right hand
corner.
\medskip

(iii) Let us reduce $B_{21}$ by adding
$\varDelta B_{21} =
ZJ_m(\lambda)-J_n(\mu)Z$. If
$\lambda\ne \mu$, then $\varDelta
B_{21}$ is arbitrary; we make
$B_{21}=0$. If $\lambda= \mu$, then
\begin{multline*}%\label{kif}
\varDelta B_{21}
=Z(J_m(\lambda)-\lambda I_m)
-(J_n(\lambda)-\lambda I_n)Z
  \\=
\begin{bmatrix}
0&z_{11}&\dots&z_{1,m-1} \\
\hdotsfor{4}\\
0&z_{n-1,1}&\dots&z_{n-1,m-1} \\
0&z_{n1}&\dots&z_{n,m-1}
\end{bmatrix}
    -
\begin{bmatrix}
z_{21}&\dots&z_{2m} \\
\hdotsfor{3}\\
z_{n1}&\dots&z_{nm} \\
0&\dots&0
\end{bmatrix};
\end{multline*}
we reduce $B_{12}$ to the form
$0^{\swarrow}$ along its diagonals
starting from the lower left hand
corner.
\medskip

(iv) Finally, reduce $B_{22}$ by adding
$\varDelta B_{22} =
TJ_m(\bar{\lambda})^{-T}-J_n(\mu)T$.
Since $|\lambda |>1$ and $|\mu|>1$,
$\varDelta B_{22}$ is arbitrary; we
make $B_{22}=0$.

\subsubsection{Pairs of\/ blocks ${\cal D}
\langle \mu \Delta_m,\nu
\Delta_n)\rangle $ with $|\mu |=|\nu
|=1$} \label{sub7}

According to Corollary \ref{cor}(ii),
we need to prove that each pair $(B,A)$
of $n\times m$ and $m\times n$ matrices
can be reduced to $(0,\:0)$ if $\mu
\ne\pm\nu $ or to exactly one pair of
the form $(0^{\nwarrow},\:0)$ if $\mu
=\pm\nu $ by adding
\[
(\mu S^*
\Delta_m+ \nu \Delta_n
R,\: \nu R^*\Delta_n
+\mu \Delta_mS),\qquad S\in
 {\mathbb C}^{m\times n},\ R\in
 {\mathbb C}^{n\times m}.
\]

Taking $R=0$ and $S=-\bar\mu
\Delta_m^{-1}A$, we reduce $A$ to $0$.
To preserve $A=0$, we must hereafter
take $S$ and $R$ such that $\nu
R^*\Delta_n +\mu \Delta_mS=0$; that is,
$ S=-\bar\mu\nu \Delta_m^{-1}
R^*\Delta_n, $ and so we can add  to
$B$ matrices of the form
\[
 \varDelta B:=\nu \Delta_n R
-\mu^2\bar\nu\Delta_n^* R
\Delta_m^{-*}\Delta_m.
\]

Write $P:=\Delta_n^* R$, then
\begin{equation*}\label{due}
\varDelta B= \bar\nu[\nu^2(\Delta_n
\Delta_n^{-*})P -\mu^2 P
(\Delta_m^{-*}\Delta_m)].
\end{equation*}
Note that
\begin{equation}\label{1x12}
\Delta_n\Delta_n^{-*}= \Delta_n\begin{bmatrix}
*&&i&1
\\
&\udots&\udots\\
i&1\\
1&&&0
\end{bmatrix}
=
\begin{bmatrix} 1&&&0
\\2i&1&&\\
&\ddots&\ddots&\\
*&&2i&1
\end{bmatrix}
\end{equation}
and
\begin{equation}\label{1x11}
\Delta_m^{-*}\Delta_m=
(\Delta_n\Delta_n^{-*})^T=
\begin{bmatrix} 1&2i&&*
\\&1&\ddots&\\
&&\ddots&2i\\ 0 &&&1
\end{bmatrix}.
\end{equation}

If $\mu\ne \pm \nu$, then $\mu^2\ne
\nu^2$, the matrices $\nu^2(\Delta_n
\Delta_n^{-*})$ and $\mu^2
(\Delta_m^{-*}\Delta_m)$ have distinct
eigenvalues, and so $\varDelta B$ is
arbitrary. We make $B=0$.

If $\mu= \pm \nu$, then
\[
\frac{1}{2i\nu}\varDelta B=
\begin{bmatrix} 0&&&0
\\1&0&&\\
&\ddots&\ddots&\\
*&&1&0
\end{bmatrix}
P-P
\begin{bmatrix} 0&1&&*
\\&0&\ddots&\\
&&\ddots&1\\ 0 &&&0
\end{bmatrix}
\]
and we reduce $B$ to the form
$0^{\nwarrow}$ along its skew diagonals
starting from the upper left hand
corner.

\subsubsection{Pairs of blocks ${\cal D}
\langle J_m(0),J_n(0)\rangle $ with
$m\ge n$} \label{sub4}

According to Corollary \ref{cor}(ii) we
need to prove that each pair $(B,A)$ of
$n\times m$ and  $m\times n$ matrices
with $m\ge n$ can be reduced to exactly
one pair of the form $(0^{\swvdash},\:
0^{\swvdash})$ if $n$ is even or
$(0^{\swvdash}+{\cal
P}_{nm},\:0^{\swvdash})$ if $n$ is odd
by adding the matrices
\begin{equation*}\label{jfr}
\varDelta A=R^*J_n(0)
+J_m(0)S,\qquad \varDelta
B^*=J_m(0)^TS+
R^*J_n(0)^T
\end{equation*}
to $A$ and $B^*$ (we prefer to reduce
$B^*$ instead of $B$).

Write $S=[s_{ij}]$ and $R^*=[-r_{ij}]$
(they are $m$-by-$n$). Then
\begin{equation*}\label{drg}
\varDelta A=
\begin{bmatrix}
 s_{21}-0&s_{22}-r_{11}&s_{23}-r_{12}
 &\dots&s_{2n}-r_{1,n-1}
 \\ \hdotsfor{5}\\
 s_{m-1,1}-0&s_{m-1,2}-r_{m-2,1}
 &s_{m-1,3}-r_{m-2,2}
 &\dots&s_{m-1,n}-r_{m-2,n-1}\\
 s_{m1}-0&s_{m2}-r_{m-1,1}&s_{m3}-r_{m-1,2}
 &\dots&s_{mn}-r_{m-1,n-1}
 \\ 0-0&0-r_{m1}&0-r_{m2}&\dots&0-r_{m,n-1}
 \end{bmatrix}
\end{equation*}
and
\begin{equation*}\label{drgs}
\varDelta B^*=
\begin{bmatrix}
 0-r_{12}&0-r_{13}
 &\dots&0-r_{1n}&0-0
    \\
 s_{11}-r_{22}&s_{12}-r_{23}&\dots&
 s_{1,n-1}-r_{2n}
 &s_{1n}-0
   \\ \hdotsfor{5}\\
 s_{m-2,1}-r_{m-1,2}&s_{m-2,2}-r_{m-1,3}&
 \dots&s_{m-2,n-1}-r_{m-1,n}
 &s_{m-2,n}-0
 \\  s_{m-1,1}-r_{m2}&s_{m-1,2}-r_{m3}&
 \dots&s_{m-1,n-1}-r_{mn}
 &s_{m-1,n}-0
 \end{bmatrix}.
\end{equation*}
Adding $\varDelta A$, we reduce $A$ to
the form
\begin{equation}\label{gu}
0^{\downarrow}:=\begin{bmatrix}
   0_{m-1,n}
\\
  *\ *\ \cdots\ *
\end{bmatrix}.
\end{equation}
To preserve this form of $A$, we must
hereafter take
\begin{equation*}\label{VRS}
s_{21}=\dots=s_{m1}=0,\qquad
s_{ij}=r_{i-1,j-1}\quad
(2\le i\le m,\ 2\le
j\le n).
\end{equation*}
Write
\[
(r_{00},r_{01},\dots,r_{0,n-1}):=
(s_{11},s_{12},\dots,s_{1n}),
\]
then
\begin{equation*}\label{drgsu}
\varDelta B^*=
\begin{bmatrix}
 0-r_{12}&0-r_{13}
 &\dots&0-r_{1n}&0-0
    \\
 r_{00}-r_{22}&r_{01}-r_{23}&\dots&
 r_{0,n-2}-r_{2n}
 &r_{0,n-1}-0
    \\
0-r_{32}&r_{11}-r_{33}&\dots&
 r_{1,n-2}-r_{3n}
 &r_{1,n-1}-0
    \\
0-r_{42}&r_{21}-r_{43}&\dots&
 r_{2,n-2}-r_{4n}
 &r_{2,n-1}-0
   \\ \hdotsfor{5}\\
0-r_{m2}&r_{m-2,1}-r_{m3}&
 \dots&r_{m-2,n-2}-r_{mn}
 &r_{m-2,n-1}-0
 \end{bmatrix}.
\end{equation*}

If $r_{ij}$ and $r_{i'j'}$ are in
distinct diagonals of $\varDelta B^*$,
then $i-j\ne i'-j'$, and so
$(i,j)\ne(i',j')$. Hence, the diagonals
of $\varDelta B^*$ have no common
$r_{ij}$, and so we can reduce the
diagonals of $B^*$ independently.

The first $n$ diagonals of $\varDelta
B^*$ starting from the upper right hand
corner are
\[\begin{matrix}
0\\
-r_{1n},\: r_{0,n-1}
   \\
-\underline{r_{1,n-1}},\:
r_{0,n-2}-r_{2n},\:
\underline{r_{1,n-1}}
   \\
-r_{1,n-2},\: r_{0,n-3}-r_{2,n-1},\:
r_{1,n-2}-r_{3n},\: r_{2,n-1}
    \\
-\underline{r_{1,n-3}},\:
r_{0,n-4}-r_{2,n-2},\:
\underline{r_{1,n-3}-r_{3,n-1}},\:
r_{2,n-2}-r_{4n},\:
\underline{r_{3,n-1}}\\
\hdotsfor{1}
\end{matrix}\]
(we underline linearly dependent
entries); adding them we make the first
$n$ diagonals of $B^*$ as in
$(0^{\swvdash})^T$.

The $(n+1){\text{st}}$ diagonal of
$\varDelta B^*$ is
\[
  \begin{cases}
(r_{00}-r_{22},\,r_{11}-r_{33},\,\dots,\,
r_{n-2,n-2}-r_{nn}) & \text{if $m=n$}, \\
(r_{00}-r_{22},\,r_{11}-r_{33},\,\dots,\,
r_{n-2,n-2}-r_{nn},\,r_{n-1,n-1})
& \text{if $m>n$.}
  \end{cases}
\]
Adding it, we make the $(n+1)\text{st}$
diagonal of $B^*$ equal to zero.

If $m> n+1$, then the
$(n+2){\text{nd}}, \dots,m{\text{th}}$
diagonals of $\varDelta B^*$ are
\[
\begin{matrix}
-\underline{r_{32}},\,r_{21}-r_{43},\,
\underline{r_{32}-r_{54}},\,
\dots,\,r_{n,n-1}\\
\hdotsfor{1}\\
-\underline{r_{m-n+1,2}},\,
r_{m-n,1}-r_{m-n+2,3},\,
\underline{r_{m-n+1,2}-r_{m-n+3,4}},\,
\dots,\,r_{m-2,n-1}
\end{matrix}
\]
Each of these diagonals contains $n$
elements. If $n$ is even, then the
length of each diagonal is even and its
elements are linearly independent; we
make the corresponding diagonals of
$B^*$ equal to zero. If $n$ is odd,
then the length of each diagonal is odd
and the set of its odd-numbered
elements is linearly dependent; we make
all elements of the corresponding
diagonals of $B^*$ equal to zero except
for their last elements (they
correspond to the stars of ${\cal
P}_{nm}^*$ defined in \eqref{hui}).

It remains to reduce the last $n-1$
diagonals of $B^*$ (the last $n-2$
diagonals if $m=n$). The corresponding
diagonals of $\varDelta B^*$ are
\[
\begin{matrix}
-r_{m2}\\
-r_{m-1,2},\,
r_{m-2,1}-r_{m3}
    \\
 -{r_{m-2,2}},\,
r_{m-3,1}-r_{m-1,3},\,
{r_{m-2,2}-r_{m4}}
   \\
-{r_{m-3,2}},\,
r_{m-4,1}-r_{m-2,3},\,
{r_{m-3,2}-r_{m-1,4}},\,
r_{m-2,3}-r_{m5}
  \\ \hdotsfor{1}
    \\
-r_{m-n+3,2},\,
r_{m-n+2,1}-r_{m-n+4,3},\,\dots,\,
r_{m-2,n-3}-r_{m,n-1}
\end{matrix}
\]
and, only if $m>n$, one more diagonal
\[
-r_{m-n+2,2},\,
r_{m-n+1,1}-r_{m-n+3,3},\,\dots,\,
r_{m-2,n-2}-r_{mn}
\]
Adding these diagonals, we make the
corresponding diagonals of $B^*$ equal
to zero. To preserve the obtained zero
diagonals, we must hereafter take
$r_{m2}=r_{m4}=r_{m6}=\dots=0$ and
arbitrary
$r_{m1},\,r_{m3},\,r_{m5},\,\dots\,$.

Recall that $A$ has the form
$0^{\downarrow}$ (see \eqref{gu}).
Since
$r_{m1},\,r_{m3},\,r_{m5},\,\dots$ are
arbitrary, we can reduce $A$ to the
form
\[
\begin{bmatrix}
   0_{m-1,n}
\\
  *\ 0\ *\ 0\ \cdots
\end{bmatrix}
\]
by adding $\varDelta A$; these
additions preserve $B^*$.

If $m=n$, then $A$ can be alternatively
reduced to the form
\[
\begin{bmatrix}
\hdotsfor{4}
\\ 0&0&\dots&0
 \\ *&0&\dots&0
 \\ 0&0&\dots&0
 \\ *&0&\dots&0
\end{bmatrix}
\]
preserving the form $0^{\swvdash}$ of
$B$.

\subsection{Off-diagonal blocks of ${\cal
D}\langle A_{\text{\rm can}}\rangle $
that correspond to summands of
$A_{\text{\rm can}}$ of distinct
types}\label{s7}

Finally, we prove the assertion (ii) of
Corollary \ref{cor} for $M$ and $N$ of
distinct types.

\subsubsection{Pairs of blocks ${\cal D}
\langle H_{2m}(\lambda ), \mu
\Delta_n\rangle $ with $|\lambda|>1$
and $|\mu |=1$} \label{sub9}

According to Corollary \ref{cor}(ii) we
need to prove that each pair $(B,A)$ of
$n\times 2m$ and $2m\times n$ matrices
can be reduced to the pair $(0,0)$ by
adding
\[
(S^*
H_{2m}(\lambda )+ \mu\varDelta _n
R, \:R^* \mu  \varDelta _n
+H_{2m}(\lambda )S),\qquad S\in
 {\mathbb C}^{2m\times n},\ R\in
 {\mathbb C}^{n\times 2m}.
\]

Reduce $A$ to $0$ by this addition with
$R=0$ and $S=-H_{2m}(\lambda )^{-1}A$.
To preserve $A=0$, we must hereafter
take $S$ and $R$ such that $R^* \mu
\Delta_n +H_{2m}(\lambda )S=0$; that
is,
\[
S=-H_{2m}(\lambda )^{-1}
R^* \mu \Delta_n.
\]
Hence, we can add  to $B$ matrices of
the form
\[ \varDelta B :=
 \mu \Delta_n R
- \bar{\mu } \Delta_n^* R
H_{2m}(\lambda )^{-*}H_{2m}(\lambda ).
\]
Write $P:=\bar{\mu} \Delta_n^* R$, then
\[
 \varDelta B={\mu}{\bar{\mu}^{-1}} \Delta_n
 \Delta_n^{-*} P
-P
\left({J}_m(\lambda )
\oplus{J}_m(\bar{\lambda })^{-T}\right).
\]
By \eqref{1x12},
${\mu}{\bar{\mu}^{-1}}$ of modulus $1$
is the only eigenvalue of
${\mu}{\bar{\mu}^{-1}} \Delta_n
 \Delta_n^{-*}$. Since $|\lambda
 |>1$, ${\mu}{\bar{\mu}^{-1}}
\Delta_n
 \Delta_n^{-*}$ and ${J}_m(\lambda )
\oplus{J}_m(\bar{\lambda })^{-T}$ have
no common eigenvalues. Thus, $\varDelta
B$ is an arbitrary matrix and we can
make $B=0$.

\subsubsection{Pairs of blocks ${\cal D}
\langle H_{2m}(\lambda),J_n(0)\rangle $
with $|\lambda|>1$} \label{sub5}

According to Corollary \ref{cor}(ii),
we need to prove that each pair $(B,A)$
of $n\times 2m$ and $2m\times n$
matrices can be reduced to $(0,\:0)$ if
$n$ is even or to exactly one pair of
the form $( 0^{\updownarrow},\:0)$ if
$n$ is odd by adding
\[
(S^*
H_{2m}(\lambda)+ J_n(0)
R,\:R^*J_n(0)
+H_{2m}(\lambda)S),\qquad S\in
 {\mathbb C}^{2m\times n},\ R\in
 {\mathbb C}^{n\times 2m}.
\]

Putting $R=0$ and
$S=-H_{2m}(\lambda)^{-1}A$, we reduce
$A$ to $0$. To preserve $A=0$, we must
hereafter take $S$ and $R$ such that
$R^*J_n(0) +H_{2m}(\lambda)S=0$; that
is, we take
\[
S=-H_{2m}(\lambda)^{-1}
R^*J_n(0).
\]
Hence we can add to $B$ matrices of the
form
\begin{align*}
 \varDelta B:=& {J}_n(0) R
-{J}_n(0)^T R
H_{2m}(\lambda)^{-*}H_{2m}(\lambda)
        \\
=& {J}_n(0)
R-{J}_n(0)^T R
\left({J}_m(\lambda)
\oplus{J}_m(\bar{\lambda})^{-T}\right).
\end{align*}

Let us partition $B$ and $R$ into
$n\times m$ blocks: $ B=[M\ N],$ $R=[U\
V]$. We can add to the blocks $M$ and
$N$ of $B$ matrices of the form
\[
\varDelta M:=
{J}_n(0)U-{J}_n(0)^TU{J}_m(\lambda),\qquad
\varDelta N:=
{J}_n(0)V-{J}_n(0)^TV{J}_m(\bar{\lambda})^{-T}.
\]

We reduce $M$  as follows. Let
$(u_1,u_2,\dots,u_n)^T$ be the first
column of $U$. Then we can add to the
first column $M_1$ of $M$ the vector
\begin{align*}
\varDelta M_1:=&
(u_2,\,\dots,\,u_n,0)^T-
\lambda
(0,u_1,\,\dots,\,u_{n-1})^T
         \\
=&
(u_2,\,u_3-\lambda
u_1,\,u_4-\lambda
u_2,\,\dots,\,u_n-\lambda
u_{n-2},\,-\lambda
u_{n-1})^T
\end{align*}
($\varDelta M_1=0$ if $n=1$). The
elements of this vector are linearly
independent if $n$ is even, and they
are linearly dependent if $n$ is odd.
We reduce $M_1$ to zero if $n$ is even,
and to the form $(*,0,\dots,0)^T$ or
$(0,\dots,0,*)^T$ if $n$ is odd. To
preserve this form in the latter case,
we must hereafter take
$u_2=u_4=u_6=\dots=0$.

Then we successively reduce the other
columns transforming $M$ to $0$ if $n$
is even or to the form
$0_{nm}^{\updownarrow}$ if $n$ is odd.

We reduce $N$ in the same way starting
from the last column.

\subsubsection{Pairs of blocks ${\cal D}
\langle \lambda \Delta_m,J_n(0)\rangle
$ with $|\lambda|=1$} \label{sub8}

According to Corollary \ref{cor}(ii),
we need to prove that each pair $(B,A)$
of $n\times m$ and  $m\times n$
matrices can be reduced to $(0,\:0)$ if
$n$ is even or to exactly one pair of
the form $(0^{\updownarrow},\:0)$ if
$n$ is odd by adding
\[
(S^*
\lambda \Delta_m+
J_n(0)R,\:R^*J_n(0)
+\lambda \Delta_mS),\qquad S\in
 {\mathbb C}^{m\times n},\ R=[r_{ij}]\in
 {\mathbb C}^{n\times m}.
\]

We reduce $A$ to $0$ by putting $R=0$
and $S=-\bar{\lambda} \Delta_m^{-1}A$.
To preserve $A=0$, we must hereafter
take $S$ and $R$ such that $R^*J_n(0)
+\lambda \Delta_mS=0$; that is, $
S=-\bar{\lambda} \Delta_m^{-1}
R^*J_n(0).$ By \eqref{1x11}, we can add
\begin{align*}
\varDelta B :=& J_n(0)R-{\lambda}^2
J_n(0)^TR \Delta_m^{-*}\Delta_m \\
       =&\begin{bmatrix}
r_{21}&\dots&r_{2m}
           \\
\hdotsfor{3}\\
r_{n1}&\dots&r_{nm}\\
0&\dots&0
\end{bmatrix}
           -{\lambda}^2
\begin{bmatrix}
0&\dots&0\\
r_{11}&\dots&r_{1m}
           \\
\hdotsfor{3}\\
r_{n-1,1}&\dots&r_{n-1,m}
\end{bmatrix}
\begin{bmatrix} 1&2i&&*
\\&1&\ddots&\\
&&\ddots&2i\\ 0 &&&1
\end{bmatrix}
\end{align*}
to $B$. We reduce $B$ to $0$ if $n$ is
even or to the form $0^{\updownarrow}$
if $n$ is odd along its columns
starting from the first column.

\section*{Acknowledgements}

A. Dmytryshyn was supported by the
Swedish Research Council (VR) under
grant A0581501, and by eSSENCE, a
strategic collaborative e-Science
programme funded by the Swedish
Research Council. V. Futorny was
supported by the CNPq (grant
301743/2007-0) and FAPESP (grant
2010/50347-9). This work was done
during two visits of V.V. Sergeichuk to
the University of S\~ao Paulo. He is
grateful to the University of S\~ao
Paulo for hospitality and to the FAPESP
for financial support (grants
2010/07278-6 and 2012/18139-2).

\end{document}